\documentclass[psamsfont,12pt]{amsart}

\usepackage{amsmath}
\usepackage{amsthm}
\usepackage{amssymb,amsfonts}
\usepackage{mathtools}
\usepackage{tensor}

\usepackage[usenames]{xcolor}
\usepackage{amsthm,amsfonts,amsmath,latexsym,epsfig,mathrsfs,marvosym}
\usepackage{hyperref}
\usepackage[english]{babel}

\DeclareMathAlphabet\oldmathcal{OMS}        {cmsy}{b}{n}
\SetMathAlphabet    \oldmathcal{normal}{OMS}{cmsy}{m}{n}
\DeclareMathAlphabet\oldmathbcal{OMS}       {cmsy}{b}{n}
\newtheorem{theorem}{Theorem}[section]
\newtheorem{lemma}[theorem]{Lemma}
\newtheorem{corollary}[theorem]{Corollary}

\newtheorem{proposition}[theorem]{Proposition}

\newtheorem{prop_def}[theorem]{Definition/Proposition}

\newenvironment{example}{\medskip \refstepcounter{theorem}
	\noindent  {\bf Example \thetheorem}.\rm}{\,}
\newtheorem*{ack}{Acknowledgements}

\theoremstyle{remark}

\theoremstyle{definition}
\newtheorem{definition}[theorem]{Definition}

\newtheorem{remark}[theorem]{Remark}

\numberwithin{equation}{section}

\def\mS{\mathcal{S}} 
\def\bfV{\mbox{vol}} \def\bfS{\mbox{{\bf S}}}  
\def\del{\partial}

\def\kt{\mathfrak{t}}
 
\def\vol{\varpi}
\def\ra{\rightarrow }

\newcommand{\R}{\mathbb{R}}
\newcommand{\Z}{\mathbb{Z}}

\def\mL{\mathcal{L}}
\def\bT{\mathbb T}
\def\fract#1#2{\raise4pt\hbox{$ #1 \atop #2 $}}

\def\M{ W}

\newcommand{\C}{\mathbb{C}}

\def\Pot{ G }
\def\pot{ u }

\newcommand{\Q}{\mathbb{Q}}
\newcommand{\dd}{\partial \bar{\partial}}
\newcommand{\p}{\partial}

\newcommand{\tx}{\tilde{x}}
\newcommand{\bl}{\hat{\ell}}
\newcommand{\tell}{\tilde{\ell}}
\newcommand{\bv}{\hat{v}}
\newcommand{\bD}{\hat{D}}

\newcommand{\betaCONE}{\mathcal{B}}

\begin{document}
	
\title[Toric Sasaki-Einstein metrics]{Toric Sasaki-Einstein metrics with conical singularities}
\author[M. de Borbon]{Martin de Borbon}
\address{LMJL, Universit\'e de Nantes }
\email{martin.deborbon@univ-nantes.fr}
\author[E. Legendre]{Eveline Legendre}
\address{Universit\'e de Toulouse}
\email{eveline.legendre@math.univ-tolouse.fr}

\begin{abstract}
		We show that any toric K\"ahler cone with smooth compact cross-section admits a family of Calabi-Yau cone metrics with conical singularities along its toric divisors. The family is parametrized by the Reeb cone and the angles are given explicitly in terms of the Reeb vector field. The result is optimal, in the sense that any toric Calabi-Yau cone metric with conical singularities along the toric divisor (and smooth elsewhere) belongs to this family. We also provide examples and interpret our results in terms of Sasaki-Einstein metrics.     
\end{abstract}

\maketitle


\section{Introduction}

K\"ahler-Einstein metrics with conical singularities along divisors are canonical differential geometric structures on pairs of algebraic varieties endowed with real coefficient divisors. More precisely, the natural algebro-geometric framework in which their theory develops, is on the setting of klt pairs. From the analytic side, there is a general existence theory of weak K\"ahler-Einstein metrics on klt pairs, see \cite{Guedj}. In particular, one is interested in describing the tangent cones of these singular K\"ahler-Einstein metrics. The theory of normalized volumes of valuations (\cite{Li}) associates, by purely algebraic methods, affine cones to klt pair singularities. It is expected that these algebraic cones agree with the tangent cones of the appropriate singular K\"ahler-Einstein metrics. This expectation has been verified in a few cases; in the absolute setting -which means no cone singularities along divisors- see \cite{Hein-Sun}. However, such results in the logarithmic/cone singularities along divisors case are rare.

In this paper, extending famous results of Martelli--Sparks--Yau~\cite{MSY_toric} and Futaki--Ono--Wang~\cite{FOW}, we endow affine toric pairs  with Calabi-Yau cone metrics with conical singularities along the invariant divisors, which serve as natural candidates for tangent cones in the toric context. In particular, the metrics that we produce, are plausible tangent cones of K\"ahler-Einstein metrics on toric log Fano pairs at isolated singular points of the ambient variety, as provided by \cite{Berman}.

\subsection{Main results}
Let \(X\) be a toric K\"ahler cone of complex dimension \(n+1\) with smooth compact cross section. In what follows, $\bT = \R^{n+1}/ 2\pi \Z^{n+1}$ denotes a compact torus of dimension $n+1$ acting efficiently and holomorphically on \(X\). As we recall in section~\ref{ss:toricFRAMEWORK}, the associated moment cone is 
\begin{equation*}
C =  \{p \in \R^{n+1} \setminus \{0\} \hspace{2mm} \mbox{s.t } \ell_a(p) \geq 0 \hspace{2mm} \mbox{for } a=1, \ldots, d \} ,
\end{equation*}
where \(\ell_a\) are the linear functions defining the facets \(F_a = \{ \ell_a =0 \} \subset C\). We take inward normals to the facets, so \(\ell_a\) is given by taking the Euclidean inner product against a primitive integer vector \(v_a \in \Z^{n+1}\),
\begin{equation*}
\ell_a = \langle \cdot, v_a \rangle .
\end{equation*} 
The Reeb cone is the interior of the dual cone \(C^*_0\), where
\begin{align*}
C^* &= \{q \in \R^{n+1} \hspace{2mm} \mbox{s.t } \langle p, q \rangle \geq 0 \hspace{2mm} \mbox{for all } p \in C \} \\
&= \{\sum_a \lambda_a v_a \hspace{2mm} \lambda_a \geq 0 \} .
\end{align*}

We have an injective linear map \(L: \R^{n+1} \to \R^d\) given by 
\begin{equation}\label{e:mapL}
p \to (\ell_1(p), \ldots, \ell_d(p)). 
\end{equation}
The \emph{angles' cone} is the image of \(L\) intersected with the positive octant, that is 
\begin{equation*}
\betaCONE = \{ \beta = (\beta_1, \ldots, \beta_d) \in \R^d_{>0}\} \cap \mbox{Image}(L) .
\end{equation*}
Equivalently, \((\beta_1, \ldots, \beta_d)\) belongs to \(\betaCONE\) if and only if there is a point \(p\) in the interior of \(C\) such that
\begin{equation}\label{anglecondition}
\beta_a = \ell_a(p) \hspace{2mm} \mbox{for } a=1, \ldots, d .
\end{equation} 
The map \eqref{e:mapL} embeds the interior of the moment cone as \(C_0 \cong \betaCONE \subset \R^d\).

As a general rule, we use \(X\) to denote the cone \emph{without} its apex. In particular, \(X\) as well as its invariant divisors \(D_a \subset X\) are smooth manifolds. We fix \((X, \omega)\) as a symplectic compact cone manifold  and consider compatible K\"ahler cone structures on it, see eg \cite{Abreu_sas, MSY_toric}. As a general fact, the corresponding compatible complex structures that we consider are all biholomorphic.
We provide a notion of a toric K\"ahler cone metric on \(X\) with cone angles \(2\pi\beta_a\) along \(D_a\) (Definition \ref{def:CONEMETRIC}) by defining a suitable class of symplectic potentials (Definition \ref{def:SYMPLPOT}). With these concepts, our main results are the following.

\begin{theorem}\label{MAINTHM}
	There is a \((n+1)\)-family of $\bT$--invariant compatible Calabi-Yau cone metrics on \((X, \omega)\) with cone singularities along its toric divisors. The family of metrics can be realized in the following two equivalent ways.
	\begin{itemize}
		\item \emph{Fixing the Reeb vector field.} For every \(\xi\) in the interior of \(C^*\) there is a unique \(\beta \in \betaCONE\) such that \(X\) has a $\bT$--invariant Calabi-Yau cone metric with cone angles \(2\pi\beta_a\) along \(D_a\) and Reeb vector equal to \(\xi\).
		
		\item \emph{Fixing the cone angles.} For every \(\beta\) in \(\betaCONE\) there is a unique \(\xi\) in the interior of \(C^*\) such that \(X\) has a $\bT$--invariant Calabi-Yau cone metric with cone angles \(2\pi\beta_a\) along \(D_a\) and Reeb vector equal to \(\xi\).
	\end{itemize}
	In either case, the Calabi-Yau cone metric with prescribed Reeb vector or cone angles is unique up to isometry.
\end{theorem}
More precisely,
the map $\xi\mapsto \beta \in \betaCONE$ of the first part of Theorem~\ref{MAINTHM} is explicitly given by $$\beta_a = \frac{1}{\int_{P_\xi} d\tx} \int_{P_\xi} \ell_a(\tx)d\tx ,$$ where $d\tx =d\tx_1\wedge \dots \wedge d\tx_n$ and $(\tx_1,\dots,\tx_n)$ are affine coordinates on the transversal polytope $P_\xi :=\{ x\in C\, | \, \langle \xi, x \rangle = 1/2\}$.

Theorem \ref{MAINTHM} admits a Sasakian reformulation that goes as follows.
Write \(S\) for the link of the cone, so \(X = C(S)\) is diffeomorphic to
\begin{equation*}
X \cong \R_{>0} \times S
\end{equation*}
and \(S\) is a Sasaki manifold of dimension \(2n+1\). Let \(\Sigma_a\) the codimension two submanifolds of \(S\) cut out by \(D_a\),  so 
\begin{equation*}
D_a = C(\Sigma_a) , \hspace{2mm} \mbox{with } \Sigma_a \subset S .
\end{equation*}
We can restate Theorem \ref{MAINTHM} in terms of the existence of toric Sasaki-Einstein metrics on \(S\) with cone angles \(2\pi\beta_a\) along \(\Sigma_a\). In particular, Theorem \ref{MAINTHM} asserts that \emph{any} toric compact Sasaki manifold admits a family of Sasaki-Einstein metrics with conical singularities in real codimension two.

Theorem \ref{MAINTHM} provides a complete family, in the sense that if there is a Calabi-Yau cone metric on \(X\) with cone angles \(2\pi\beta_a\) along \(D_a\), then the cone angles must necessarily satisfy Equation \eqref{anglecondition} and the metric must be isometric to one given by Theorem \ref{MAINTHM}.
The next result translates our condition $\beta\in \betaCONE$ given by Equation \eqref{anglecondition} in cohomological terms. 

\begin{theorem} \label{THM2} Let $(X,\omega)$ be a toric K\"ahler cone with smooth compact cross section.    
	The following statements are equivalent:
	\begin{enumerate}
		\item The cone angle condition given by Equation \eqref{anglecondition} holds, i.e. $\beta\in \betaCONE$.
		\item There is a compatible toric Calabi-Yau cone metric on \((X, \omega)\) with cone angles \(2\pi\beta_a\) along \(D_a\).
		\item For any compatible toric K\"ahler cone metric on $(X, \omega)$ with cone angles \(2\pi\beta_a\) along \(D_a\), there exists a smooth function $h$ on \(X\) such that the associated Ricci form satisfies $\rho^\omega= i\dd h$ on $X\setminus \cup_a D_a$.  
		\item \((\bar{X}, \sum_a(1-\beta_a)\bar{D}_a)\) is a klt log pair, where \(\bar{X} = X \cup \{o\}\), \(\bar{D}_a = D_a \cup \{o\}\) and \(\{o\}\) is the apex of the cone.
		\item \(c_1(H) = \sum_a (1-\beta_a)[\Sigma_a]\) and  \(c_1^B - \sum_a (1-\beta_a)[\Sigma_a]_B >0\), where \(H \subset TS\) is the contact distribution, \([\Sigma_a] \in H^2(S, \R)\) are the Poinar\'e duals of the smooth toric submanifolds \(\Sigma_a \subset S\) and \(c_1^B\) is the basic first Chern class.
	\end{enumerate}
\end{theorem} 

Recall that  \(\bar{X} = X \cup \{o\} \) has the structure of an affine toric algebraic variety, homeomorphic to
\begin{equation*}
\R_{\geq 0} \times S / (\{0\} \times S) .
\end{equation*}
Its apex \(\{o\}\) is characterized as being the only fixed point of the torus action and the variety \(\bar{X}\) is singular at \(\{o\}\), unless \(\bar{X} = \C^{n+1}\). The definition of a klt log pair will be recalled in Section \ref{sect:ALG}; it involves two conditions, which can be paired with their corresponding Sasakian analogues as follows
\begin{align*}
K_{\bar{X}} + \sum_a (1-\beta_a)\bar{D}_a  \hspace{1mm}\mbox{ is \(\R\)-Cartier} &\leftrightsquigarrow c_1(H) = \sum_a (1-\beta_a)[\Sigma_a] \\
\mbox{the log discrepancies are \(>-1\)} &\leftrightsquigarrow c_1^B - \sum_a (1-\beta_a)[\Sigma_a]_B >0 .
\end{align*}

From the algebraic point of view, the angle constraints given by Equation \eqref{anglecondition} is equivalent to \(K_{\bar{X}} + \sum_a (1-\beta_a)\bar{D}_a\) being \(\R\)-Cartier. On the other hand, from the Sasakian perspective, Equation \eqref{anglecondition} is equivalent to the vanishing of the `logarithmic' first Chern class \(c_1(H) - \sum_a (1-\beta_a)[\Sigma_a]\). In the toric case we consider, it is a general fact that if \(K_{\bar{X}} + \sum_a (1-\beta_a)\bar{D}_a\) is \(\R\)-Cartier then the log pair \((\bar{X}, \sum_a (1-\beta_a)\bar{D}_a)\) is automatically klt. Similarly, if \(c_1(H) = \sum_a (1-\beta_a)[\Sigma_a]\), then the basic logarithmic first Chern class \(c_1^B - \sum_a (1-\beta_a)[\Sigma_a]_B\) is automatically positive. The klt property of the pair/positivity of the logarithmic basic first Chern class are natural necessary assumptions in the search of conically singular Calabi-Yau metrics, as provided by Theorem \ref{MAINTHM}.

\subsection{Transversal polytopes, barycenters, Reeb vector fields and cone angles}

The proof of Theorem~\ref{MAINTHM} is a combination of the principal result of \cite{RKE_legendre} and an adaptation of the variational characterization of the vanishing of the transversal Futaki invariant \cite{MSY_toric}. In the recent years, thanks to the work of many people, e.g. \cite{Abreu_sas, H2FII, don:Large, guillMET, MSY_toric}, a precise dictionary between toric K\"ahler geometry and convex affine geometry over some convex polyhedral set ({\it rational polytopes}) has been established. This allows us to translate any torus invariant geometric PDE problem on a toric polarized complex manifold into a boundary value problem on a convex polytope $P$ in an affine space $H$. The boundary conditions depend on a {\it labelling} $\tell = (\tell_1,\dots,\tell_d)$, that is a minimal set of fixed affine linear functions such that $P=\{ x \in H \,|\, \tell_a(x)\geq 0, \,\ a=1,\dots, d\}.$

An asset of this point of view is that given a Reeb vector, i.e. a toric polarized symplectic cone $(X, \omega,\xi)$ as above, we get a natural labelled polytope 
\begin{equation}\begin{split}
P_\xi :=&\{ x\in \R^{n+1}\,|\, \langle \xi, x\rangle =1/2 \mbox{ and } \beta_a^{-1}\ell_a(x) \geq 0, \,\ a=1,\dots, d\}\\
&\qquad \subset \{ x\in \R^{n+1}\,|\, \langle \xi, x\rangle =1/2\} =: H_\xi ;
\end{split} 
\end{equation} which is not necessarily rational but on which the PDE analogue to the K\"ahler--Einstein problem makes sense and its resolution, given by Wang--Zhu~\cite{WangZhu} in the smooth, non conical, compact toric case, still holds. Along these lines, a straightforward application of \cite[Theorem 1.6 and \S 6]{RKE_legendre} is:\\

{\it There exists a toric Calabi-Yau cone metric with Reeb vector field $\xi$ and conical singularities of angles $2\pi\beta_a$ along the toric divisors $D_a$ if and only if the transversal polytope $P_\xi\subset C$ satisfies the following combinatorial condition:} 
\begin{itemize}
	\item[$(\star)_\beta$] {\it The barycenter (with respect to any affine measure) of $P_\xi$, say $p\in P_\xi$, satisfies} 
	$$\beta_1^{-1}\ell_1(p)= \cdots =\beta_d^{-1}\ell_d(p) = (n+1)^{-1}.$$
\end{itemize}   

It turns out that the condition given by \eqref{anglecondition} is exactly the right one to run the variational principle established by Martelli--Sparks--Yau~\cite{MSY_toric} in our generalized setting. Indeed, given $\beta\in \betaCONE$, it is the image of a unique point $(n+1)p_\beta \in C_0$ via the map \eqref{e:mapL} and we can define the set $\Xi_\beta$ of Reeb vector fields whose transversal polytope $P_\xi$ contains $p_\beta$. That is $$\Xi_\beta := \{ \xi \in C^*_0\,|\, \langle \xi, p_\beta\rangle =1/2\},$$  which is a cross section of $C^*_0$ and a convex open polytope in the affine hyperplane defined by $\langle \cdot , p_\beta\rangle =1/2$. Then, extending the main result of Martelli--Sparks--Yau~\cite{MSY_toric}, we prove
\begin{lemma}                                                   
	The volume functional $$\bfV : \Xi_\beta \ra \R $$ is strictly convex and its unique minimum $\xi_\beta\in \Xi_\beta$ is the only Reeb vector field in $C^*_0$ satisfying the condition $(\star)_\beta$. In particular, $\xi_\beta$ is the only Reeb vector field in $C^*_0$ admitting a compatible toric Calabi-Yau cone metric with cone angles $2\pi\beta_a$ along the toric divisors $D_a$.    
\end{lemma}

The family of metrics given by Theorem \ref{MAINTHM} can be parametrized by either the Reeb vector field \(\xi\) or by the cone angles \(\beta\) along the toric divisors, thus providing a bijection between \(C^*_0\) and \(\betaCONE \cong C_0\). This illustrates a natural bijection between the interior of a cone and the interior of its dual cone highlighted by Gigena in \cite[Lemma 3.1]{Gigena}. The maps \(\xi \to \beta\) and \(\beta \to \xi\) are inverses of each other, are \((-1)\)-homogeneous, and define an analytic bijection between the  \(C^*_0\) and \(\betaCONE\). In particular, it follows from the algebraic character of the volume function, that the components of \(\xi(\beta)\) are algebraic numbers over \(\Q(\beta)\) and vice-versa.

\begin{ack}
	The first named author was financially supported by the ANR Project CCEM, \textbf{ANR-17-CE40-0034}. The second named author was supported by the CNRS COOPINTEER \textbf{IEA-295351}.
\end{ack}

\section{Toric conical K\"ahler metrics in action-angle coordinates} 

\subsection{Basic example}\label{s:basicEX}

We start by looking at \(\C^{n+1}\), we begin first with \(\C\).
Let \(\beta>0\) and write \(\C_{\beta} = (\C, g_{\beta})\) for the complex numbers endowed with the singular metric \(g_{\beta} = \beta^2 |z|^{2\beta-2}|dz|^2\). Writing \(z = r^{1/\beta}e^{i\theta}\), we have
\[g_{\beta} = dr^2 + \beta^2 r^2 d\theta^2  \]
which we recognize as the cone over a circle of radius \(2\pi \beta\). Its Reeb vector field is 
\begin{equation*}
I(r \p_r) = \beta^{-1} \p_{\theta} 
\end{equation*}
and the symplectic form is \(\omega_{\beta} = \beta r dr \wedge d\theta \). The vector field \(\p_{\theta}\) generates a circle action by holomorphic isometries, with Hamiltonian
\begin{equation*}
x = \frac{\beta r^2}{2} .
\end{equation*}
In action angle coordinates \((x, \theta)\), we have \(\omega_{\beta} = dx \wedge d\theta\) and
\begin{align*}
g_{\beta} &= \frac{1}{2\beta x} dx^2 + 2\beta x d\theta^2 \\
&= G''dx^2 + (G'')^{-1} d\theta^2 ,
\end{align*}
with symplectic potential
\begin{equation*}
G = \frac{1}{2\beta} x \log x.
\end{equation*}

Consider next the product of two cones \(dr_1^2 + \beta_1^2 r_1^2 d\theta_1^2\) and \(dr_2^2 + \beta_2^2r_2^2d\theta_2^2\). We introduce variables \((r, \psi) \in (0, \infty) \times (0, \pi/2) \) defined by \(r_1 = r \cos(\psi)\) and \(r_2 = r \sin(\psi)\). It is  then easy to check that  \(\C_{\beta_1} \times \C_{\beta_2} \) is a cone, in the sense that \(g_{\C_{\beta_1} \times \C_{\beta_2}} = dr^2 + r^2 g_{S^3_{(\beta)}}\), with link
\begin{equation*}
g_{S^3_{(\beta)}}  = d\psi^2 + \beta_1^2\cos^2(\psi)d\theta_1^2 + \beta_2^2 \sin^2(\psi) d\theta_2^2 .
\end{equation*}
We see that \(g_{S^3_{(\beta)}}\) defines a Sasaki metric on the three-sphere, with constant sectional curvature \(1\) and cone angles \(2\pi\beta_1\) and \(2\pi\beta_2\) along the Hopf circles defined by intersecting \(\{z_1=0\}\) and \(\{z_2=0\}\) with the unit three-sphere.

Same way, the metric product \(\mathbf{C}_{\beta_1} \times \ldots \times \mathbf{C}_{\beta_{n+1}}\) defines a Calabi-Yau (indeed flat) cone metric on \(\C^{n+1} \) with cone angles \(2\pi\beta_a\) along \(\{z_a=0\}\). Its link is the \((2n+1)\)-sphere, endowed with a Sasaki-Einstein metric of constant sectional curvature \(1\) and cone angles \(2\pi\beta_a\) along the \((2n-1)\)-spheres cut out by \(\{z_a=0\}\cap S^{2n+1}\).
Its Reeb vector field is  
\begin{equation*}
\frac{1}{\beta_1} \frac{\p}{\p \theta_1} + \ldots + \frac{1}{\beta_{n+1}}  \frac{\p}{\p \theta_{n+1}} .
\end{equation*} 
The diagonal action of \(\bT\) on \(\C_{\beta_1} \times \ldots \times \C_{\beta_{n+1}}\) is by Hamiltonian isometries. 
We have action angle coordinates \((x, \theta)\) with components \(x_i = \beta_i r_i^2 /2\). The moment cone is \( \cap_i \{x_i>0\} \subset \R^{n+1} \) and the symplectic potential is
\begin{equation*}
G = \frac{1}{2} \sum_a \beta_a^{-1} x_a \log x_a .
\end{equation*}
The Reeb vector/cone angles correspondence of Theorem \ref{MAINTHM} in this case is simply 
\begin{equation*}
\xi (\beta) = (\beta_1^{-1}, \ldots, \beta_{n+1}^{-1}) .
\end{equation*}

\subsection{Symplectic potentials} \label{ss:toricFRAMEWORK} \label{sect:SYMPLPOT} The goal of this section is to describe K\"ahler toric metrics with conical singularities in terms of their symplectic potentials. To this end we need to review a bit action-angle and complex coordinates on toric K\"ahler manifolds.  

\begin{remark}
	Since most of the content of the paper is differential-geometric, we use \(X\) to denote the cone without its apex.
\end{remark}

\subsubsection{Symplectic potentials of smooth K\"ahler toric metrics} 

Let $(X^{n+1}, J, \omega)$ be a smooth K\"ahler cone over a compact base with radial vector field, classically denoted $r \p_r \in \Gamma(TX)$. Recall that it means that $r\p_r$ induces a free, holomorphic and proper action of $\R^+$ and $\omega$ is homogeneous of order $2$ (i.e $\mL_{r\p_r} \omega=2\omega$) and that $X/\R^+ =:S$ is compact. The {\it Reeb vector field} is, by definition, $\xi := Jr\p_r$. 

In what follows we assume moreover that $(X, J, \omega)$ is \emph{toric} meaning that a $(n+1)$--dimensional compact torus $\bT$ acts effectively on $X$ in a Hamiltonian fashion and holomorphically. In that case, there is a unique momentum map $\mu : X\ra \kt^*$, homogeneous of order $2$ with respect to the $\R^+$--action where $\kt=\mbox{Lie}(\bT)$ and \(\kt^*\) is its dual. According to \cite{Lerman}, the \emph{moment cone} $C = \mu(X)$ is then a strictly convex rational polyhedral cone (without its apex) that can be written as $$C=\{ x\in \kt^* \setminus\{0\} \,|\, \langle x, \ell_a \rangle \geq 0, \, a=1,\dots, d\}$$ where $d$ is the number of facets of $C$ and $2\pi\ell_1,\dots, 2\pi\ell_d \in \Lambda \subset \kt$ are primitives in the lattice of circle of subgroups of $\bT = \kt / \Lambda$.
 
\begin{remark}
	In most of the paper, we identify \(\kt \cong \R^{n+1}\) and \(\Lambda \cong 2\pi \Z^{n+1}\).  We also use the Euclidean inner product to get \(\kt \cong \kt^* \cong \R^{n+1}\).
\end{remark}

The fact that the cross section of the cone \(C(S)\) is a smooth manifold, amounts to \(C\) being \emph{good} in the sense of Lerman \cite{Lerman}. We let \(\kt \cong \kt^* \cong \R^{n+1}\), then  we have primitive inward normals  \(v_a \in \Z^{n+1}\) such that
\(\ell_a (\cdot) = \langle v_a, \cdot \rangle\), where \(\langle \cdot, \cdot \rangle\) is the Euclidean inner product. 
The good condition on \(C\) means that, for any face \(F = \cap_{A=1}^N \{\ell_{a_A} =0\}\), the following holds
\[\{\sum_{A=1}^{N} \nu_A v_{a_A}, \,\ \nu_A \in \R \} \cap \Z^{n+1} = \{\sum_{A=1}^{N}\nu_A v_{a_A}, \,\ \nu_A \in \Z\} .\] 
Conversely, any good strictly convex rational polyhedral cone is the moment cone of a smooth toric K\"ahler cone.

Let \(C_0\) be the interior of \(C\), so \begin{equation}\label{eq:AAcoord} X_0 = \mu^{-1}(C_0) \cong C_0 \times \bT\end{equation}
is the set of points where the action is free \cite{Lerman}. The coordinates on $X_0$ given by the r.h.s. of \eqref{eq:AAcoord} are called the {\it action-angle coordinates}, see \cite{CDG, guillMET, MSY_toric}. Locally, it gives coordinates $(x_0,\dots,x_n,\theta_0,\dots, \theta_n)\in \R^{2n+2}\cong \kt^*\times \kt$ on $X_0$ where the class $[\theta]$ -with \(\theta_i \sim \theta_i + 2\pi\)- parametrizes \(\bT\). Guillemin~\cite{guillMET} proved that any $\bT$--invariant K\"ahler structure $(\omega,J,g)$ in action--angle coordinates \((x, [\theta])\) (thus on $X_0$), is of the form
$$\omega = \sum_{i=0}^n dx_i \wedge d\theta_i, $$  
with $g=g_G$ and $J=J_G$ given by  
\begin{equation} \label{eq:METSP}
	g_G := G_{ij} dx_i \otimes dx_j + G^{ij}d\theta_i \otimes d\theta_j,
\end{equation}
\begin{equation} \label{eq:CXSP}
	J_G \p_{x_i} = G_{ij} \p_{\theta_j}, \hspace{2mm} J_G \p_{\theta_i} = -G^{ij} \p_{x_j} .
\end{equation}
 Here, \(G_{ij} = \p^2 G / \p x_i \p x_j\) are the entries of the Hessian of a smooth strictly convex function \(G : C_0\ra \R\) and \(G^{ij}\) are the entries of the inverse matrix. It is common to call $G$ a {\it symplectic potential} of $g_G$. Observe that we need action-angle coordinates (therefore momentum map and a fixed symplectic form) to interpret $G$ or $g_G$ as a metric on $X_0$. The resulting metric does not depend on the choice of these coordinates see \cite{H2FII}.

The condition for a smooth strictly convex function \(G : C_0\ra \R\) to define a smooth K\"ahler cone metrics are well-known as we recall now.

 \begin{proposition}\cite{{abreuOrbifold},H2FII, guillMET, MSY_toric} \label{prop:SYMPPOT}
 The tensor \eqref{eq:METSP} extends as a smooth compatible K\"ahler cone metric on $(X,\omega)$  if and only if   
	\begin{itemize}
		\item[(i)] the restriction of $\Pot$ to the interior of $C$ -and to the interior of any of its faces of positive dimension- is smooth, strictly convex and such that 
		\begin{equation}\label{eq:Guilcond} \Pot -\frac{1}{2}\sum_{a=1}^d \ell_a \log \ell_a \in C^{\infty}(C); 	 
		\end{equation}
		\item[(ii)] the Hessian of $\Pot$ is homogeneous of order $-1$ with respect to the natural $\R^+$--action on $\kt^*$ (i.e radial vector field $\sum_{i=0}^{n} x_i\frac{\del }{\del x_i}$). 
	\end{itemize}
	In that case, the Reeb vector field $\xi = (\xi_0, \ldots, \xi_n) \in C^*_0 \subset \kt\cong  \R^{n+1}$ of $g_G$ is given in coordinates as $$\xi_j=2\Pot_{ij} x_i.$$
\end{proposition}
 
 \begin{definition}\label{def:SYMPLPOTsmooth}
  	We denote $$\mS_\xi(C,\ell):=\{G : C_0\ra \R \,|\, G \mbox{ satisfies } (i), (ii) \mbox{ and } \xi_j=2\Pot_{ij} x_i \mbox{ as in Proposition } \ref{prop:SYMPPOT} \}$$  the space of $\xi$--\emph{symplectic potentials} on $C$ with respect to $\ell_1,\dots, \ell_d$. 
 \end{definition}

A nice fact (see again \cite{guillMET, CDG, MSY_toric}) is that  \( (x, \theta) \mapsto (y =DG (x), \theta) \) are complex coordinates on the tangent space, meaning that $z=y+i\theta$ are local holomorphic coordinates on $X_0$. If we let   
$$z_k = \log w_k = y_k + i \theta_k , $$ 
then $(w_0,\dots w_n)$ are global coordinates on $X_0\cong (\C^*)^{n+1}$. The Legendre transform 
\(F\) of \(G\), defined by
\[F(y) + G(x) = \langle x, y \rangle,  \hspace{3mm} \forall (x,y)\in C_0\times \kt  ,\] 
is a {\it K\"ahler potential} of $\omega$. Indeed
\[\omega =\sum_{k=0}^n dx_k\wedge d\theta_k = \sum_{k=0}^n d (\frac{\partial F}{\partial y_k})\wedge d\theta_k = \sum_{j, k=0}^n \frac{\partial^2 F}{\partial y_j\partial y_k}dy_j \wedge d\theta_k  =2 i \dd F  ;\]
as follows from the identity $x= DF (y)$.

\begin{remark}
Given another symplectic potential $\tilde{G} \in  \mS_\xi(C,\ell)$, the map $$\Phi_{G\tilde{G}}:=(D\tilde{G})^{-1} \circ (D{G}) : C_0\ra C_0$$ extends as a smooth $\bT$--equivariant diffeomorphism of $X$ such that $ \Phi_{G\tilde{G}}^* J_G = J_{\tilde{G}}$ and  
$\Phi_{G\tilde{G}}^* \omega=  \omega + 2i\dd f$ where $f$ is a smooth homogeneous (degree \(0\)) $\bT$--invariant function on $X$, see~\cite{H2FII, don:Large, MSY_toric}. This is one way to get that any two $\bT$--invariant smooth complex structures on a toric symplectic cone manifold $(X,\omega)$ are biholomorphic.
\end{remark}

\begin{prop_def} The {\it Reeb cone} $C^*_0$ is the interior of the dual moment cone, that is $C^*_0:=\{\xi\in \kt \,|\, \langle \xi, x \rangle >0,\, \forall x\in C\}$. Any smooth toric K\"ahler cone metric on $(X,\omega)$  has a Reeb vector field lying in $C^*_0$ and any $\xi\in C^*_0$ is the Reeb vector field of a smooth toric K\"ahler cone metric on $(X,\omega)$.   
\end{prop_def}

\subsubsection{Toric K\"ahler metric with conical singularities along a divisor}

For any $a\in\{1,\dots, d\}$ we denote the divisor $D_a:=\mu^{-1}(\ell_a^{-1}(0)\cap C )$ in $X$.  
Recall that, as proved in~\cite[\S 6.3]{RKE_legendre}, for any fixed $a\in\{1,\dots, d\}$, the tensor $g_G$ given by Equation \eqref{eq:METSP} extends as a metric on $X_0\cup D_a$ with cone angle $2\pi\beta_a$ along the divisor $D_a$ provided that 
\begin{equation} \label{eq:CONESING}
	G- \frac{1}{2} \beta_a^{-1}\ell_a \log \ell_a \; \in C^\infty\left(C_0 \cup \left(\ell_a^{-1}(0)\setminus \cup_{b \neq a} \ell_b^{-1}(0)\right) \right).	
\end{equation}
Precisely, assuming that $$G=\frac{1}{2} \beta_a^{-1}\ell_a \log \ell_a +f$$ is convex on \(C_0 \cup \left(\ell_a^{-1}(0)\setminus \cup_{b \neq a} \ell_b^{-1}(0)\right)\) with $f$ smooth on $C_0 \cup \left(\ell_a^{-1}(0)\setminus \cup_{b \neq a} \ell_b^{-1}(0)\right)$, then $$g_G = \frac{1}{2\beta_a x_a} dx_a^2 + 2\beta_a x_a d\theta_a^2  + C^{\infty} = \beta_a^2 |z_0|^{2(\beta_a-1)}dz_0\wedge d\bar{z}_0 +C^\infty . $$ 
Here, up to a linear change of coordinates, we put $x_a:=\ell_a(x)$, \(\theta_a =\langle \ell_a,\theta\rangle\) and $z_0= (2x_a/\beta_a)^{1/2\beta_a}e^{i\theta_a}$. The \(C^{\infty}\) term on the r.h.s. is understood as smooth in terms of \((|z_0|^{\beta_a-1}z_0, z_1, \ldots, z_n)\). The same comments apply at the intersection of two or more divisors \(D_{a_1} \cap \ldots \cap D_{a_k}\) for \(1 \leq k \leq n\). 

In this paper, we  say that the toric metric \(g_G\) has cone angle \(2\pi\beta_a\) along \(D_a\) if its symplectic potential satisfies the smoothness condition in Equation \eqref{eq:CONESING}. In practice, when proving existence theorems, one has to relax the smoothness condition (see \cite{RKE_legendre}, \cite{SongWang}) and consider \(C^{k, \alpha}\) metrics. At the end, regularity results show that these weaker notions of conical singularities satisfy the smoothness condition provided they solve the K\"ahler-Einstein equations on the complement of the conical divisors. 

\begin{remark}
	Note that for making sense of $g_G$ as a  metric on $X_0$ we are using the action-angle coordinates of a fixed symplectic form $\omega$ which admits a smooth extension to $X$. 
	There are two differential structures on \(X\), meaning that there are two atlas; one in which our fixed symplectic form extends smoothly and the other in which the complex structure does. The last differential structure depends on the cone angles \(\beta=(\beta_1, \ldots, \beta_d)\). These two differential structures are equivalent, by diffeormorphisms modelled on \(re^{i\theta} \to r^{1/\beta}e^{i\theta}\) in transverse directions to the conical divisors. In the paper, when we say smooth, we mean smooth with respect to the atlas of our fixed symplectic manifold.
\end{remark}

\begin{definition}\label{def:CONEMETRIC} Let $(X,\omega)$ be a toric symplectic cone whose moment cone $(C,\ell)$ is strictly convex. A compatible toric K\"ahler metric $g$ with cone angles $2\pi \beta_a$ along the divisors $D_a$ is a toric K\"ahler cone metric on $X_0 \cong C_0\times \bT^{n+1}$ such that for each $a\in\{1,\dots, d\}$ $$g - \left( \frac{1}{2\beta_a \ell_a(x)} d\ell_a^2 + 2\beta_a \ell_a(x) d\theta_a^2\right) ,$$
where \(\theta_a=\langle \ell_a,\theta\rangle\), extends smoothly (in the symplectic sense) on $X_0\cup \left(D_a \setminus \cup_{b \neq a}D_b\right)$ and restricts in a positive definite tensor on $TD_a$. Similarly, for \(1 \leq k \leq n\) and \(I=\{a_1, \ldots, a_k\} \subset \{1, \ldots, d\} \) with \(D_{a_1} \cap \ldots \cap D_{a_k}\) non-empty, we require
$$g - \sum_{j=1}^{k} \left( \frac{1}{2\beta_{a_j} \ell_{a_j}(x)} d\ell_{a_j}^2 + 2\beta_{a_j} \ell_{a_j}(x) d\theta_{a_j}^2\right)$$
to extend smoothly over $X_0\cup \left( \cup_{a \in I} D_a \setminus \cup_{b \neq I}D_b\right)$.	
\end{definition}

By the observation above and putting together the work of many people (in the compact K\"ahler case \cite{abreu, {abreuOrbifold}, H2FII, don:scalar,don:extMcond,don:Large, guillMET, RKE_legendre}, made clear in the K\"ahler cone case by \cite{MSY_toric}, see also~\cite{Abreu_sas,nonUNIQcscS}) we get that

 \begin{proposition}\label{prop:Conical=symp}\label{p:corresPOT} Let  \(\beta_1, \ldots, \beta_d \in \R_{>0}\). For $G\in C^0(C)\cap C^\infty(C_0)$, the tensor $g_G$ of \eqref{eq:METSP} is a toric K\"ahler cone metric with conical singularities of angles $2\pi\beta_a$ along the divisors $D_a$ if and only if: 
 	\begin{itemize}
		\item[(i)] the restriction of $\Pot$ to the interior of $C$, and to the interior of any of its faces of positive dimension, is smooth and strictly convex and such that 
		\begin{equation}\label{eq:CDTbetasmooth}\Pot -\frac{1}{2}\sum_{a=0}^d \beta_a^{-1} \ell_a \log \ell_a \, \in C^\infty(C) ;
		\end{equation} 
		\item[(ii)] the Hessian of $\Pot$ is homogeneous of order $-1$ with respect to the natural $\R^+$--action on $\kt^*$ (i.e radial vector field $\sum_{i=0}^{n} x_i\frac{\del }{\del x_i}$).
	\end{itemize} In that case the Reeb vector field $\xi = (\xi_0, \ldots, \xi_n) \in C^*_0 \subset \kt\cong  \R^{n+1}$ of $g_G$ is given in coordinates as $$\xi_j=2\sum_{i=0}^n\Pot_{ij} x_i. $$  
\end{proposition}

	This prompts us to define the following set of functions 
 \begin{definition}\label{def:SYMPLPOT}
  	We denote $$\mS_\xi(C,\ell,\beta):=\{G \in C^0(C)\cap C^\infty(C_0) \,|\, G \mbox{ satisfies } (i), (ii)\mbox{ of Proposition } \ref{prop:Conical=symp} \mbox{ and } \xi_j=2\Pot_{ij} x_i \}$$  the space of $\xi$--\emph{symplectic potentials} on $C$ with respect to $\ell_1,\dots, \ell_d$ and cone angles \(2\pi\beta_1, \ldots, 2\pi\beta_d\).
 \end{definition}

\begin{remark}\label{r:NaturalLABELLING}
	In Definition \ref{def:SYMPLPOT} we have distinguished the natural integer labelling \(\ell_1, \ldots, \ell_d\) -determined by the geometric (complex, algebraic or symplectic) structure of \(X\)- from the cone angles \(\beta_1, \ldots, \beta_d\). Alternatively, one can mix these, by letting  \(\tilde{\ell}_a = \beta_a^{-1} \ell_a \) for \(a=1, \ldots, d\). This way \(\tilde{\ell}_1, \ldots, \tilde{\ell}_d\) gives another labelling of \(C\), and we will also write \(\mS_{\xi}(C, \tilde{\ell}) = \mS_{\xi}(C, \ell, \beta)\). 
\end{remark}

We define the $\beta$--Guillemin potential as $$G^{can} :=  \frac{1}{2}  \sum_a \beta_a^{-1} \ell_a\log\ell_a $$ and we observe that $G^{can} \in \mS_{\xi^{can}}(C,\ell,\beta)$ for \[\xi^{can} = \sum_a \beta_a^{-1} \ell_a . \] 
For later use, we record the following
\begin{equation} \label{eq:GUILLp1}
\frac{\p}{\p x_i}  G^{can}= \frac{1}{2} \sum_a \beta_a^{-1} (1+\log \ell_a)v^a_i , 
\end{equation}
\begin{equation} \label{eq:GUILLp2}
\frac{\p^2}{\p x_i \p x_j} G^{can} = \frac{1}{2} \sum_a \beta_a^{-1} v^a_i v^a_j \ell_a^{-1} .
\end{equation}

\begin{remark} In the case where $\beta_1 =\cdots= \beta_d=1$, it was shown in \cite{guillMET}, see also \cite{CDG}, that $G^{can}$ is the symplectic potential of the toric K\"ahler metric obtained as the K\"ahler reduction of the flat metric on $\C^d$ through the  Delzant construction. It seems natural to think that for general angles $\beta \in \R^d_{>0}$, the $\beta$--Guillemin potential is the symplectic potential of the toric K\"ahler metric obtained as the K\"ahler reduction of the flat metric on $\C_\beta^d$ of \S \ref{s:basicEX} this way.    
\end{remark}

Let us set
\[\ell_{\infty} = \sum_a \beta_a^{-1} \ell_a \]
and, for \(\xi \in C^*_0\), 
\[G_{\xi} = \frac{1}{2} \ell_{\xi} \log \ell_{\xi} - \frac{1}{2} \ell_{\infty} \log \ell_{\infty} . \]
We introduce a canonical symplectic potential with cone angles \(2\pi\beta_a\) along \(D_a\) and prescribed Reeb vector, 
\begin{equation} \label{Geq}
G^{can}_{\xi} := G^{can} + G_{\xi} \in \mathcal{S}_{\xi}(C, \ell, \beta) .
\end{equation} The conditions $(i)$ and $(ii)$ defining $\mathcal{S}_{\xi}(C, \ell, \beta)$ translate as 
\begin{corollary}\label{p:MSYprop} $\mS_\xi(C,\ell, \beta)$ coincides with the subspace of strictly convex functions on the interior of $C$ (and interior of faces of $C$) that can be written as $$G = \frac{1}{2}\sum_{a=1}^d \beta_a^{-1} \ell_a \log \ell_a - \frac{1}{2}\ell_\infty \log \ell_\infty + \frac{1}{2}\xi\log \xi  + f $$ where $\ell_\infty =\sum_{a=1}^{d} \beta_a^{-1} \ell_a$ and $f\in C^{\infty}(C)$ is homogeneous of order $1$ with respect to the radial vector field $\sum_{i=0}^n x_i\frac{\del}{\del x_i}$.   
\end{corollary}

Symplectic potentials in \(\mS_{\xi}(C, \ell, \beta)\) correspond to symplectic potentials on compact cross section labelled polytopes \(P_{\xi}\). This will be crucial in our application. We recall the following notion.

\begin{definition} \label{d:lapPOLsympPOT} Given a convex compact polytope $P$ in an affine space $H$ given by $P:=\{x\in H\,|\, \tilde{\ell}_1(x)\geq 0,\dots,\tilde{\ell}_d(x)\geq 0 \}$, a \emph{symplectic potential} on $P$ with respect to $\tilde{\ell}_1,\dots, \tilde{\ell}_d$ is a continuous function $\pot\in C^0(P)$ such that its restriction to the interior of $P$, or any of its faces, is smooth and convex and such that $$\pot -\frac{1}{2}\sum_{a=1}^d \tilde{\ell}_a \log \tilde{\ell}_a$$ is smooth on $P$. We denote $\mS(P,\tilde{\ell})$ the space of \emph{symplectic potentials} on $P$ with respect to $\tilde{\ell}_1,\dots, \tilde{\ell}_d$.  
\end{definition}   
 Note that, as linear functions on $\kt^*$,  $\beta_1^{-1}\ell_1,\dots, \beta_d^{-1}\ell_d \in\kt$ define affine functions on the affine hyperplane $\{x\in \kt^* \,|\, \langle x,\xi \rangle =1/2 \}$ which plays the role of $H$ in the Definition~\ref{d:lapPOLsympPOT}.
\begin{corollary} \label{c:RESTpotential}
A $\xi$--symplectic potential $\Pot\in\mS_{\xi}(C, \ell, \beta) $ on $C$ is uniquely determined by its restriction to $P_\xi:=C\cap \{x\in \kt^* \,|\, \langle x,\xi \rangle =1/2 \}$, which turns out to be a symplectic potential on $P_\xi$ with respect to the labelling induced by $\beta_1^{-1}\ell_1,\dots, \beta_d^{-1}\ell_d$. 
\end{corollary}

\begin{remark} For any $\beta\in \R_{>0}^d$ and potential $G\in \mS_\xi(C, \ell, \beta)$, $(g_G,\omega)$ is a $\kt$--invariant  K\"ahler cone structure on $(C_0\times \kt, \omega)$  with Reeb vector field $\xi\in \kt$ over the non compact Sasakian manifold obtained as the link $\{r=1\}\simeq \mathring{P_\xi}\times \kt$ where $r^2= g(\xi,\xi)$. Of course, any local computation in the K\"ahler cone/Sasaki context is still valid here and thus, for example $$\omega = \frac{i}{2}\dd r^2$$ which implies that $r^2/4$ is, up to addition by an affine-linear function of the complex coordinates $y=d_x\Pot$, the Legendre transform of $\Pot$. 
In this situation, there is no preferred lattice and no notion of quasi-regular Reeb vector field but we always have a (non-compact) K\"ahler reduction  $(\mathring{P_\xi}\times (\kt/\R \xi), \check{\omega}, \check{J})$ so that $$\mathring{P_\xi}\times (\kt/\R \xi) = x^{-1}(1/2)/ \R \xi = \{r=1\}/ \R \xi$$ and the pull back of $\check{\omega}$ on $\{r=1\}$ is the restriction of $\frac{i}{2} \dd r^2$ which coincides with $i \dd \log r^2$ on $\{r=1\}$.  
\end{remark}

\subsection{Scalar and Ricci curvature of symplectic potentials and log Futaki invariant}\label{ss:ScalFUT}

Each Reeb vector field $\xi \in C_0^*$, determines a hyperplane  \(H_{\xi} = \{2 \langle \xi, x \rangle =1 \}\) and a corresponding polytope $P_\xi$ which is labelled by \(\beta_1^{-1} \ell_1, \dots,  \beta_d^{-1} \ell_d\). In this section we recall the {\it Futaki invariant of the labelled polytope} $(P_\xi,\beta_1^{-1} \ell_1, \dots,  \beta_d^{-1} \ell_d)$ which provides an obstruction to the existence of scalar-flat metric on $(X,\omega,\xi)$ with conical angle $2\pi\beta_a$ along $D_a$, \cite{FOW, nonUNIQcscS, RKE_legendre}. Equivalently, this is an obstruction to the existence of scalar-flat potential in $\mS_\xi(C,\beta_1^{-1} \ell_1, \dots,  \beta_d^{-1} \ell_d)$. Note that this is just a convex affine translation of the classical Futaki invariant \cite{futaki}. It agrees with the {\it log Futaki invariant}, which arises in the more general setting of metrics with cone singularities along divisors (not necessarily toric) see ~\cite{hashimoto}.

In the next statement, we denote $(\Pot^{ij})$ the inverse Hessian of a potential $\Pot \in \mS_\xi(C,\ell,\beta)$. It is a smooth matrix valued function on $C$ by \cite{H2FII} and we recall the so-called Abreu formula which expresses the scalar curvature of the metric associated to a potential $G$ on \(X\) as  $$R_X = - \sum_{i,j=0}^n \frac{\del^2 \,\Pot^{ij}}{\del x_i \del x_j} .$$
The scalar curvature \(R_X\) can be, and will be, seen as a smooth function on $C$ or sometimes identified with its pull back on $X$, as a smooth $\bT$--invariant function. In the next claim, we state straightforward consequences of the computations in \cite{Abreu_sas}, see also \cite{nonUNIQcscS,MSY_toric}, which hold on $X_0$ and then apply directly on the singular toric K\"ahler metric of Definition \ref{def:CONEMETRIC}.     

\begin{proposition}\label{p:abreuFORMULA} A toric K\"ahler cone $(X^{n+1}, J, \omega)$ with Reeb vector field $\xi$ and conical singularity of angle $2\pi\beta_a$ along the divisor $D_a$ is scalar flat if and only the corresponding potential $\Pot$ on $C$ satisfies \begin{equation}\label{abreuFORMULAcone}R_X=-\sum_{i,j=0}^n \frac{\del^2 \,\Pot^{ij}}{\del x_i \del x_j}  =0\end{equation} and this happens if and only if the corresponding potential $\pot\in \mS(P_\xi, \ell)$ on $P_\xi$ satisfies \begin{equation}\label{abreuFORMULA}
	-\sum_{i,j=1}^{n} \frac{\del^2\, \pot^{ij}}{\del \tx_i \del \tx_j}  = n(n+1), \end{equation}
where \((\tx_1, \ldots, \tx_n)\) are coordinates con the affine hyperplane \(H_{\xi} = \{2 \langle \xi, x \rangle =1 \}\).	
\end{proposition}
Formula \eqref{abreuFORMULA} is the usual Abreu formula with a specific constant $n(n+1)$.

 Via the correspondence between metrics and symplectic potentials and thanks to Proposition~\ref{p:abreuFORMULA}, it makes sense to define \emph{scalar-flat symplectic potential} (respectively \emph{csc symplectic potential}) as a potential satisfying Equation \eqref{abreuFORMULAcone} (respectively \eqref{abreuFORMULA}). It also makes sense to say that a {\it symplectic potential is K\"ahler-Einstein}, as the Ricci form $\rho^\Pot$ of the associate metric $g_G$ on $C_0\times \bT$, see \eqref{eq:AAcoord}, is $$\mbox{Ric}(g_G)= \sum_{i,j,k=1}^{n} \frac{\del^2\, \Pot^{ij}}{\del x_i \del x_k} dx_j\wedge d\theta_k$$ with respect to action-angle coordinates. Using the computations in \cite{Abreu_sas} we can therefore specialize the latter Proposition to Ricci-flat metrics on $X$ as follows.   

\begin{proposition}\label{p:abreuFORMULAricci} A toric K\"ahler cone $(X^{n+1}, J, \omega)$ with Reeb vector field $\xi$ and conical singularity of angle $2\pi\beta_a$ along the divisor $D_a$ is Ricci-flat if and only the corresponding potential $\pot\in \mS(P_\xi,\ell,\beta)$ on $P_\xi$ is a \emph{K\"ahler--Einstein potential} with scalar curvature $n(n+1)$, that is satisfies \begin{equation}\label{e:KEpot}
	-\sum_{i=1}^{n} \frac{\del^2\, \pot^{ij}}{\del \tx_i \del \tx_k}  = (n+1)\delta_{jk}, \end{equation}
where \((\tx_1, \ldots, \tx_n)\) are coordinates con the affine hyperplane \(H_{\xi} = \{2 \langle \xi, x \rangle =1 \}\).	
\end{proposition}

Next, we note that the integral of the scalar curvature of a toric Sasakian metric with conical singularities, depends only on the Reeb vector and the cone angles.
The link \(S= X|_{r=1}\) corresponds, under the moment map, to 
\[P_{\xi} := \{\langle \xi, x\rangle =1/2 \} \cap C . \]
Similarly, \(\mu (X_1) = \Delta_{\xi} \), with \(X_1 = X|_{r \leq 1}\) and \(\Delta_{\xi}\) is the  polytope
\[\Delta_{\xi} = \{ x \in C \hspace{2mm} \mbox{s.t. } 2\langle \xi, x \rangle \leq 1 \} . \]

We let \(D_a\) be the divisor corresponding to \(\{\ell_a=0\}\). Set \(\Sigma_a \subset S\), so \(D_a = C(\Sigma_a)\). Also, let \(F_a = \Delta_{\xi} \cap \{\ell_a=0\}\). It is straightforward (see \cite{guillMET}) to show that
\[\mbox{vol}(S) = 2(n+1) (2\pi)^{n+1} \mbox{vol}(\Delta_{\xi}), \hspace{2mm} \mbox{vol}(\Sigma_a) = 2n(2\pi)^{n} |v_a|^{-1} \mbox{vol}(F_a) , \]
where on the left hand sides the volume is computed with respect to the toric contact $\eta_\xi$ form induced by $\xi$ and $\omega$ on $S$ and on the right hand side the volume and the norm are computed with respect to the standard Lebesgue measure on Euclidean space.     

\begin{proposition}
	We have
	\begin{equation} \label{eq:INTSCAL}
		\int_{\Delta_{\xi}} R_X dx = \frac{2\pi}{n} \sum_a \beta_a \mbox{vol}(\Sigma_a) - 2(n+1) \mbox{vol}(S) .
	\end{equation}
	In particular, if \(R_X=0\), then
	\begin{equation} \label{voleq}
	\pi \sum_a \beta_a \mbox{vol}(\Sigma_a) = n (n+1) . \mbox{vol}(S)
	\end{equation}
\end{proposition}

\begin{proof}
	We follow \cite{MSY_toric}, incorporating the boundary conditions given by the cone angles. We integrate by parts and use Equation \eqref{eq:GUILLp1} to get
	\begin{align*}
	\int_{\Delta} R_X dx &= \sum_a \int_{F_a} G^{ij}_i v_j^a |v_a|^{-1} - \int_{P_{\xi}} G^{ij}_i b_j |b|^{-1} d\sigma_{\xi,1} \\
	&= \sum_a 2 \beta_a |v_a|^{-1} \mbox{vol}(F_a) - \frac{2(n+1)}{|b|} \mbox{vol}(H) \\
	&= \frac{2\pi}{n} \sum_a \beta_a \mbox{vol}(\Sigma_a) - 2(n+1) \mbox{vol}(S) .
	\end{align*}
\end{proof}

\begin{example} Let \(S\) be the three-sphere equiped with the metric \(g_{S^3_{(\beta)}}\) (Secrion \ref{s:basicEX}), which constant curvature \(1\) and two Hopf circles with cone angles \(2\pi\beta_1\), \(2\pi\beta_2\) of respective lengths \(2\pi\beta_2\) and \(2\pi\beta_1\). We have \(\mbox{vol}(S) = \beta_1 \beta_2 2\pi^2\) and Equation \eqref{voleq} reads
\[\pi (\beta_1 2\pi \beta_2 + \beta_2 2\pi \beta_1 ) = 2 \beta_1 \beta_2 2 \pi^2 .  \]
\end{example}

We link Equation \eqref{voleq} with a previous observation of \cite{don:scalar} and introduce the (transversal) log Futaki invariant.  As shown in \cite{Abreu_sas}, the scalar curvature of a Sasaki metric $g_\pot$,  associated to $\pot\in  \mS(P_\xi,\beta_1^{-1}\ell_1,\dots, \beta_d^{-1}\ell_d)$, on the link $S= \{z\in X\,|\, \langle \mu(z), \xi\rangle =1/2\}$ is the pull back of 
$$\mbox{scal}_{g_\pot} = -2n - 4\sum_{i,j=0}^n \frac{\del^2 \,\pot^{ij}}{\del x_i \del x_j} = R_X + 2n(2n+1).$$  A simple but useful calculation of \cite[Lemma 3.3.5]{don:scalar} using the boundary conditions satisfied by $\pot\in  \mS(P_\xi,\beta_1^{-1}\ell_1,\dots, \beta_d^{-1}\ell_d)$ shows that, what we call the {\it total transversal scalar curvature} is  
\begin{equation}\label{e:totScal=intBND}
 \bfS(\xi) := \frac{1}{4(2\pi)^n}\int_{S} (\mbox{scal}_{g_\pot} +2n )d\mbox{vol}_{g_\pot}=2\sum_a \int_{\partial P_\xi}\, \sigma_{\xi,\beta} = 2 \sum_a \beta_a \int_{F_a}\, \sigma_{\xi, 1} 
\end{equation} where $\sigma_{\xi,\beta}$ is the volume form on $\partial P_\xi$ defined by \begin{equation}\label{e:conMEASUREbndy}
\beta_a^{-1} d\ell_a \wedge \sigma_{\xi,\beta} = - d\tx_1\wedge \dots \wedge d\tx_n \qquad \qquad \mbox{ on the facet } F_a\cap P_\xi ;
\end{equation} where the fibers of the cotangent space of $\kt^*$ are naturally identified to $\kt$ and $(\tx_1,\dots, \tx_n)$ are affine coordinates on the hyperplane $H_\xi$. Actually, the calculation \cite[Lemma 3.3.5]{don:scalar} shows more, namely that for any $\pot \in \mS(P_\xi,\beta_1^{-1}\ell_1,\dots, \beta_d^{-1}\ell_d)$, we have 
$$\int_{P_\xi}\, \left(-\sum_{i,j=1}^n\pot^{ij}_{ij}\right) f(\tx) d\tx =2 \int_{\partial P_\xi}\, f(\tx) \sigma_{\xi,\beta},$$
for any affine function $f$. Which in turn implies that the $L^2(P_\xi, d\tx)$--projection of the transversal scalar curvature $$-\sum_{i,j=1}^n\pot^{ij}_{ij}$$ on the space of affine functions, does not depend on the potential $\pot\in \mS(P_\xi,\beta_1^{-1}\ell_1,\dots, \beta_d^{-1}\ell_d)$ but only on $P_{\xi}$ and the labelling $\beta_1^{-1}\ell_1,\dots, \beta_d^{-1}\ell_d$. We call this affine function $A_\beta$. We have $A_\beta(\tx)= A_0 + \sum_{i=1}^nA_i\tx_i$ on $P_\xi$ with the constants \(A_j\) given by  
\[ \sum_{j=0}^{n} \left( \int_{P_{\xi}} \tx_i \tx_j d\tx \right) A_j = 2 \int_{\p P_{\xi}} \tx_i \sigma_{\xi, \beta}, \hspace{2mm} \mbox{for } i=0, \ldots, n ;\] in particular $A_0 = \frac{\int_{\p P_{\xi}}  \sigma_{\xi, \beta}}{\int_{P_{\xi}} d\tx }$.  The (transversal) log Futaki invariant is the linear function on \(\kt\cong \mbox{Aff}(H_\xi,\R)\) given by
\begin{equation}\label{eq:LogFut}
	L_{\xi,\beta}(f) = \int_{\p P_{\xi}} f \sigma_{\xi, \beta} - \frac{\int_{\p P_{\xi}}  \sigma_{\xi, \beta}}{\int_{P_{\xi}} d\tx } \int_{P_\xi} f d\tx.
\end{equation} We recall the following well-known obstruction. 
\begin{corollary}
The log Futaki invariant $L_{\xi,\beta}$ vanishes if and only if \(A_\beta \equiv A_0\). Equivalently, the log Futaki invariant vanishes if and only if the barycenter of \((P_{\xi}, d\tx)\) agrees with the barycenter of \((\p P_{\xi}, \sigma_{\xi, \beta})\).  In particular, if there exists a compatible scalar-flat K\"ahler cone metric on $(X, \omega)$ with Reeb vector field $\xi$ and conical singularities of angle $2\pi\beta_a$ along $D_a$ then $L_{\xi,\beta}\equiv 0$.
\end{corollary}

 Of course the (log) Futaki invariant is an obstruction to constant scalar curvature Sasaki metrics not only those for which $R_X\equiv 0$. 
 \begin{remark}
  Whenever $\xi$ is regular, $P_\xi$ is a Delzant polytope associated to the smooth toric symplectic reduction of $(X,\omega)$ at the level $1/2$. In that case, it is easy to show that \eqref{eq:LogFut} is, up to a positive dimensional constant, the log Futaki invariant of \cite{hashimoto}. 
 \end{remark}

\section{Proof of Theorem \ref{MAINTHM}}
The first part of Theorem \ref{MAINTHM} is proved in Corollary~\ref{c:MainReebfixed} below and the second part follows from Corollary~\ref{c:MainAnglefixed}.

\subsection{Fixing the Reeb vector field} \label{sect:FIXREEB}
Summarizing the discussion of \S\ref{sect:SYMPLPOT}, \S\ref{ss:ScalFUT} and in particular Proposition~\ref{p:abreuFORMULAricci}, the search of Ricci-flat toric K\"ahler cone metric with conical singularities along the invariant divisors boils down to the study of symplectic potentials on (labelled) polytopes solving equation~\ref{e:KEpot}. This last problem has been studied in \cite{RKE_legendre} for compact simple polytope after a suggestion made in \cite{don:extMcond}. We recall the main result.

\begin{proposition} \label{p:KEiffMONOTONE+Fut=0} \cite{RKE_legendre} Let $P$ be a simple compact convex polytope with $d$-facets (ordered $F_1,\dots F_d$) and labelling $\ell_1,\dots,\ell_d$. Given a set $\beta_1,\dots, \beta_d \in \R^+$, $\mS(P,\beta_1^{-1}\ell_1,\dots, \beta_d^{-1}\ell_d)$ contains a KE symplectic potential if and only if there exists $\lambda>0$ and $p\in \mathring{P}$ such that $\lambda \beta_a = \ell_a(p)$ for $a=1,\dots,d$ where $p$ is the barycenter of $P$, up to an overall homothety. 
\end{proposition}

In \cite{RKE_legendre}, the last statement is expressed in terms of {\it monotone labelled polytopes}, which are defined just below. It says that a labelled polytope $(P,\ell)$ admits a KE potential in $\mS(P,\ell)$ if and only if it is monotone and its log Futaki invariant vanishes.

\begin{definition} \label{d:MONOTpol}
	A labelled compact polytope $(P,\tilde{\ell})$ is monotone if and only if there exist $p\in \mathring{P}$ and $\lambda >0$ such that $\tilde{\ell}_1(p) =\cdots= \tilde{\ell}_d(p)=\lambda.$ 
\end{definition}

\begin{remark}\label{r:monotone} To explain why the monotone condition is a necessary condition, we recall that a compact symplectic orbifold $(\M,\omega)$ is called monotone if $\exists \lambda >0$ such that $$\lambda [\omega] = 2\pi c_1(\M)$$ in de Rham cohomology. In particular this condition is necessary to have a $\omega$--compatible K\"ahler--Einstein metric of positive scalar curvature. In the toric setting, $(\M,\omega)$ is monotone if and only the associated labelled polytope $(P,\tilde{\ell})$ is monotone in the sense of Definition~\ref{d:MONOTpol}. The proof of this fact works for non-Delzant labelled polytope, see eg~\cite{don:Large} and \cite[Lemma~2.4]{LTF}, as it amounts to compare the Ricci potential and the K\"ahler potential as functions on the moment polytope. That is, to check if there exist $p\in \kt^*$, $\lambda >0$ such that \begin{equation}\label{eq:Ric-omega}
	x\mapsto -\frac{1}{2} \log \det (u_{ij})_x - \lambda(\langle x -p, d_x u\rangle -u(x))   
	\end{equation} is smooth on $P$ for some (and then any) $u\in \mS(P,\tilde{\ell})$. 
In Proposition~\ref{prop:RICPOT}, we give a cone version of this monotone condition. 
\end{remark}

Proposition \ref{p:KEiffMONOTONE+Fut=0} applies directly to K\"ahler cone metrics thanks to Corollary~\ref{c:RESTpotential} and Proposition~\ref{p:abreuFORMULAricci} and we get the following existence result where we denote for a facet $F_a =\{\ell_a^{-1}(0)\}\cap C$, the corresponding divisor $D_a\subset X$.

\begin{corollary}\label{c:MainReebfixed} Let $(X^{n+1}, J, \omega)$ be a toric K\"ahler cone with Reeb vector field $\xi$ and moment cone $C$ with primitive inward normals $\ell_1,\dots, \ell_d \in \Z^{n+1}$. There are $\beta_1,\dots, \beta_d \in \R^+$, unique up to an overall homothety, such that $\mS(P_\xi,\beta_1^{-1}\ell_1,\dots, \beta_d^{-1}\ell_d)$ contains a KE symplectic potential.  In particular, given any Reeb vector field $\xi$ there exists a unique Ricci-flat K\"ahler cone metric on $X$ with conical singularities along $\cup D_a$ (but the angles might be greater than $2\pi$). 
\end{corollary}

This finishes the proof of the first item in Theorem \ref{MAINTHM}. 

\begin{remark}
	The family of metrics in Theorem \ref{MAINTHM}, parametrized by the Reeb cone, has a natural `Weil-Petersson' type metric. We can associate for any \(\xi \in C^*_0\) a conically singular Sasaki-Einstein metric \(g_{SE}(\xi)\). Up to a constant dimensional factor, the total volume of \(g_{SE}(\xi)\) is equal to \(\mbox{vol}(\Delta_{\xi})\). The Hessian of the convex function 
	\[\xi \to \log (\mbox{vol} (\Delta_{\xi})) \]
	defines a canonical metric on the Reeb cone (\cite[Appendix A.1]{Oda}), hence on our parametrized family of Calabi-Yau cone metrics.
\end{remark}

\subsection{Fixing the cone angles} \label{sect:FIXANGLE}

Consider a toric K\"ahler cone $(X^{n+1}, J, \omega_0)$  with moment cone $C$ and primitive inward normals $\ell_1,\dots, \ell_d \in \Z^{n+1}$. Let $\beta_1,\dots, \beta_d \in \R^+$ be a set of positive real numbers. We wonder if there exists a Ricci-flat $\bT$--invariant toric K\"ahler cone metric $\omega$ which is smooth on $X\backslash \cup_a D_a$ and has conical singularities of angle $2\pi\beta_a$ along $D_a$. That is, according to Proposition~\ref{p:abreuFORMULAricci}, if there exists a KE potential in $\mS_\xi(C,\beta_1^{-1}\ell_1,\dots, \beta_d^{-1}\ell_d)$ for some $\xi\in\kt_+$. 

\begin{lemma}\label{l:monotone} 
If there exists $\xi\in\kt_+$ such that $\mS_\xi(C,\beta_1^{-1}\ell_1,\dots, \beta_d^{-1}\ell_d)$ contains a KE potential then there exists a ray $E_\beta \subset C$ such that $\forall p\in E_\beta$, $$\beta_1^{-1} \ell_1(p) = \cdots = \beta_d^{-1} \ell_d(p).$$ In particular, $\beta=(\beta_1,\dots,\beta_d) \in \betaCONE$.
\end{lemma}
\begin{proof} If there exists a KE potential in $\mS(P_\xi,\beta_1\ell_1,\dots, \beta_d\ell_d)$ it implies \cite{don:Large,RKE_legendre} that $(P_\xi,\beta_1^{-1}\ell_1,\dots, \beta_d^{-1}\ell_d)$ is monotone in the sense of Definition~\ref{d:MONOTpol}. That is, there exists $p\in P_\xi$ such that $$\beta_1^{-1} \ell_1(p) = \cdots = \beta_d^{-1} \ell_d(p).$$ This condition holds for any point in the ray passing through $p\in P_\xi\subset C$ by linearity. \end{proof}
Given a point $p$ lying in the interior of $C$ the set of angles $\beta_a= \ell_a(p)$ satisfy the claim of the last Lemma. It is easy to see that any set of angles satisfying the condition is obtained this way. We conclude
\begin{corollary} The set of angles $\beta\in \R_+^d$ satisfying the condition of Lemma~\ref{l:monotone} is a $(n+1)$--dimensional cone which agrees with the angles' cone \(\betaCONE \cong C_0\) defined in the Introduction.
\end{corollary}

Since $\beta\in \betaCONE$, there is a ray $E_\beta \subset C$ and a constant $\lambda_\xi$ depending on $\xi\in C_0^*$ such that $$\beta_1^{-1}\ell_1(p) = \cdots = \beta_d^{-1} \ell_d(p) =\lambda_\xi $$ where $p$ is the single point lying in $P_\xi \cap E_\beta$. Denote $p=(p_1,\dots, p_n)$ the expression of $p$ in the coordinates $(\tx_1,\dots, \tx_n)$ of the hyperplane $H_\xi:= \{x\in \kt^*\,|\,\langle \xi,x \rangle =1/2\}$ and put 
\begin{equation}\label{e:MEASmonotone}
\sigma_{\xi,\beta} :=\frac{1}{\lambda_\xi} \sum_{i=1}^n (-1)^{i+1} (\tx_i-p_i) d\tx_1\wedge \dots \wedge \widehat{d\tx_i}\wedge \dots \wedge d\tx_n
\end{equation} where $\widehat{d\tx_i}$ means that $d\tx_i$ is omitted.

\begin{lemma} Assume that $\beta\in \betaCONE$. Then $\sigma_{\xi,\beta}$ as defined in \eqref{e:MEASmonotone} satisfies Equation \eqref{e:conMEASUREbndy}.
\end{lemma}
\begin{proof} Pick $a\in\{1,\dots, d\}$, restricted on $H_\xi$, $\ell_a$ is an affine linear function so that $\beta_a^{-1}\ell_a(\tx)= c_a + \langle \beta_a^{-1} d\ell_a, \tx\rangle $ for some constant $c_a$. Identifying the fibers of the cotangent space of $\kt^*$ to $\kt$, we have $\beta_a^{-1} d\ell_a = \sum_{i=1}^n s_{a,i} d\tx_i$ for some constants $s_{a,i}$. For $\tx\in F_a\cap P_\xi$, we have 
	\begin{equation}\label{e:calsigma} \begin{split}
	\beta_a^{-1} d\ell_a \wedge \sigma_{\xi,\beta} &=  \frac{1}{\lambda_\xi} \sum_{i=1}^n  s_{a,i}(\tx_i-p_i) d\tx_1\wedge \dots \wedge d\tx_n\\
	&= \frac{1}{\lambda_\xi}\langle \beta_a^{-1} d\ell_a, (\tx -p)\rangle d\tx_1\wedge \dots \wedge d\tx_n\\
	&= \frac{1}{\lambda_\xi}\left(\beta_a^{-1} \ell_a( \tx)   - \beta_a^{-1} \ell_a( p)\right) d\tx_1\wedge \dots \wedge d\tx_n\\
	&= - d\tx_1\wedge \dots \wedge d\tx_n
	\end{split}
	\end{equation} using that $\beta_a^{-1} \ell_a( p)=\lambda_\xi$ for any $a\in\{1,\dots, d\}$. \end{proof}
Recall that the volume of the Sasaki metric $g_\phi$,  associated to $\phi\in  \mS(P_\xi,\beta_1^{-1}\ell_1,\dots, \beta_d^{-1}\ell_d)$, on the link $N_\xi= \{z\in X\,|\, \langle \mu(z), \xi\rangle =1/2\}$ is given by $(2\pi)^n \bfV(\xi)$ where
$$\bfV(\xi) := \frac{1}{(2\pi)^n}\int_{N_\xi} \vol_{g_\phi} = \int_{P_\xi} d\tx_1\wedge \dots \wedge d\tx_n.$$

Using Stokes' Theorem we have 
\begin{lemma} \label{lem:NORMALIZATION} Assume that $\beta\in \betaCONE$. Then $\sigma_{\xi,\beta}$ as defined in \eqref{e:MEASmonotone} is globally defined on the hyperplane $H_\xi$ and $$d\sigma_{\xi,\beta}= \frac{n}{\lambda_\xi} d\tx_1\wedge \dots \wedge d\tx_n.$$ In particular $\bfS(\xi) = 2\frac{n}{\lambda_\xi}\bfV(\xi)$. 
\end{lemma}

We are interested in the case $\bfS(\xi) = 2n(n+1)\bfV(\xi)$ because we want K\"ahler--Ricci-flat cones. By the latter Lemma \ref{lem:NORMALIZATION} this coincides with the condition $\lambda_\xi = (n+1)^{-1}$. Therefore, we pick the unique $q_\beta\in E_\beta$ such that $$\beta_1^{-1}\ell_1(q_{\beta}) = \cdots = \beta_d^{-1}\ell_d(q_{\beta}) = \frac{1}{n+1}$$ and define the set of Reeb vector fields $\xi\in C^*_0$ whose polytope $P_\xi$ contains $q_\beta$, that is 
\begin{equation}\label{e:conLAMBDA}
\Xi_\beta :=\{ \xi\in C^*_0 \,|\, \langle \xi , q_\beta \rangle = 1/2\}. 
\end{equation} Pick $\xi \in \Xi_\beta$, then $\Xi_\beta =C^*_0\cap \{ \xi + \nu \,|\, \langle \nu , q_\beta \rangle = 0 \}$ and so $\Xi_\beta$ is an open convex set in the hyperplane $\{ \xi + \nu \,|\, \langle \nu , q_\beta \rangle = 0 \}$. 
Observe also that $$\Xi_\beta  = \{ \xi\in C^*_0 \,|\, P_\xi \cap E_\beta = \{ q_\beta \}\} = \{ \xi\in C^*_0 \,|\, \lambda_\xi = (n+1)^{-1}\}.$$ In particular, for $\xi \in \Xi_\beta$ we have $\bfS(\xi) = 2n(n+1)\bfV(\xi)$ and  
\begin{equation}\label{e:RateBAR}
\int_{\partial P_\xi} \sigma_{\xi,\beta}=n(n+1)\int_{P_\xi} d\tx_1\wedge \dots \wedge d\tx_n.   
\end{equation}
This will be used in the proof of the following Lemma.
\begin{lemma} Assume that $\beta\in \betaCONE$. Then the volume functional $\bfV : \Xi_\beta \ra \R$ is convex and proper and its unique critical point $\xi\in \Xi_\beta$ is characterized by the condition that the barycenter of $P_\xi$ coincides with the one of $(\partial P_\xi,  \sigma_{\xi,\beta})$. In particular, the critical point of $\bfV$ in $\Xi_\beta$ is the unique $\xi \in \Xi_\beta$ with vanishing transversal Futaki invariant. \end{lemma}

\begin{proof} Let $\xi_t= \xi+t \nu$ be a small path in $\Xi_\beta$, so that $\nu\in \kt$ is such $\langle \nu , q_\beta \rangle = 0$. As before we use that any $\ell \in \kt$ defines naturally, by restriction, an affine linear on the affine hyperplane on $H_\xi \subset \kt^*$ in which lies $P_\xi$, so we may write $\ell(\tx)$ for $\tx\in H_\xi$ to emphasis that it is not linear in $\tx\in H_\xi$. We pick again coordinates $(\tx_1,\dots, \tx_n)$ on $H_\xi$.
	
	One easily shows, see \cite{MSY_toric,MSY}, that $$\left(\frac{d}{dt} \bfV(\xi_t)\right)_{t=0} = -(n+1)\int_{P_\xi} \nu(\tx) \, d\tx_1\wedge \dots \wedge d\tx_n.$$
	It readily infers that $\xi \in \mbox{crit } \bfV$ if and only if $q_\beta$ is the barycenter of $(P_\xi,\vol)$. Indeed, as a linear function on $H_\xi$, since $\langle \nu , q_\beta \rangle = 0$, $\nu$ is linear in $\tx-q_\beta$.

	Writing $ \vol =d\tx_1\wedge \dots \wedge d\tx_n$, $\sigma =\sigma_{\xi,\beta}$ and $q_\beta =(q_{\beta,1}, \dots, q_{\beta,n})$, we have for $i=1,\dots,n$
	\begin{equation}\begin{split}
	\int_{P_\xi} \tx_i \, \vol &=   \frac{\lambda_\xi}{n}\int_{P_\xi}\tx_i d \sigma \\
	&= \frac{\lambda_\xi}{n} \int_{\partial P_\xi} \tx_i \sigma -  \frac{\lambda_\xi}{n} \int_{P_\xi} d\tx_i \wedge  \sigma\\
	&= \frac{\lambda_\xi}{n} \int_{\partial P_\xi} \tx_i \sigma -  \frac{1}{n } \int_{P_\xi} (\tx_i-q_{\beta,i} ) \vol.
	\end{split}
	\end{equation} Now using \eqref{e:RateBAR} and $\lambda_\xi = (n+1)^{-1}$ because $\xi \in \Xi_\beta$, we get that 
	\begin{equation}
	\left( 1+ \frac{1}{n}\right)   \frac{1}{\int_{P_\xi} \vol}  \int_{P_\xi} \tx_i \,\vol =  \frac{1}{\int_{ \partial P_\xi} \sigma} \int_{\partial P_\xi} \tx_i \sigma + \; \frac{q_{\beta,i}}{n}
	\end{equation} using \eqref{e:RateBAR}. This shows that $q_\beta$ is the barycenter of $(P_\xi,\vol)$ if and only if $\mbox{bar}(P_\xi,\vol)=\mbox{bar}(\partial P_\xi,\sigma)$.
	
	Moreover, it is straightforward to check that $$\left( \frac{d^2}{dt^2} \bfV(\xi_t)\right)_{t=0} = (n+1)(n+2) \int_{P_\xi} \nu(\tx)^2 \vol$$ which proves that $\bfV$ is convex on $\Xi_\beta$. The properness of $\bfV$ along the boundary of the cone $C_0^*$, see~\cite{MSY_toric, Oda}, implies the properness on $\Xi_\beta$.  \end{proof}

\begin{remark}
	The map \(\xi \to \mbox{vol}(\Delta_{\xi})\), is also known as the \emph{characteristic function} of \(C^*\). Up to a constant dimensional factor, it is also given by
	\[\xi \to \int_C e^{-\langle \xi, x \rangle} dx . \]
	The convexity and properness along the boundary are classical results, see \cite[Proposition A.10]{Oda}. 
\end{remark}

Combined with the existence result in \cite{RKE_legendre} we have the next
\begin{corollary} If $\beta\in \betaCONE$, then there exists one and only one Reeb vector field $\xi \in C_0^*$ such that $\mS_\xi(C,\beta_1^{-1}\ell_1,\dots,\beta_d^{-1}\ell_d)$ contains a KE potential.
\end{corollary}

Using the correspondence recalled in Proposition~\ref{p:corresPOT} it yields to the following
\begin{corollary}\label{c:MainAnglefixed}
	Let $(X^{n+1}, J, \omega)$ be toric K\"ahler cone (over a compact link) with labelled cone $(C,\ell_1,\dots,\ell_d)$. There exists a $(n+1)$--dimensional set of angles $\beta\in \betaCONE \subset \R_+^d$ for which there is a unique (up to scale and isometry) Calabi-Yau cone metric $\omega(\beta)$, compatible with the complex structure $J$, which is smooth on the open dense set where the torus action is free and has cone angles $2\pi\beta_a$ along the divisors $\ell_a^{-1}(0)$. Moreover, the Reeb vector field of $\omega(\beta)$ is the unique critical point of the volume functional in the set of $\beta$--normalized Reeb vector fields $\Xi_\beta \subset \kt_+$.  
\end{corollary}

We have now completed the proof of Theorem \ref{MAINTHM}.

\begin{example}
	As explained in the Introduction, our conically singular Calabi-Yau cone metrics provide tangent cone models for singular K\"ahler-Einstein metrics on toric klt pairs. We see how this works in the toy example of projective cones (\cite[Section 3.2]{LiLiu}).
	It is a general fact, that if \(\omega_C = (i/2) \dd r^2 \) is a Calabi-Yau cone, then \(\omega_{KE} := (i/2) \dd \log(1+r^2)\) defines a K\"ahler-Einstein metric with positive Ricci and its tangent cone at the apex is the given Calabi-Yau cone, \(T_0(\omega_{KE}) = \omega_C\). The metric \(\omega_{KE}\) has finite diameter. If \(\omega_{C}\) is quasiregular then \(\omega_{KE}\) can be compactified by introducing a suitable cone angle along a copy of the base as a divisor at infinity, see \cite[Proposition 3.3]{LiLiu}. In the toric irregular case we can still think of \(\omega_{KE}\) as a solution to the K\"ahler-Einstein equation on the polytope \(\Delta_{\xi}\).
\end{example}

\section{Proof of Theorem \ref{THM2}}

In this section we provide a geometric interpretation of the cone angle constraints defining the angles' cone \(\betaCONE\), leading to the proof of Theorem \ref{THM2}.
If Equation \ref{anglecondition} holds, then Theorem \ref{MAINTHM} guarantees the existence of a Calabi-Yau metric on \(X\) with cone angles \(2\pi\beta_a\) along \(D_a\), so \((1) \implies (2)\) in Theorem \ref{THM2}. Clearly, \((2) \implies (3)\). From Proposition \ref{prop:RICPOT} below we get that \((3) \implies (1)\), so we conclude that the first three items in Theorem \ref{THM2} are equivalent.  The equivalence \((1) \iff (4)\)  follows from Proposition \ref{CartierProp} and Proposition \ref{kltProp}. Similarly, \((1) \iff (5)\) follows from Proposition \ref{logc1PROP} and Proposition \ref{c1basicPROP}.

\subsection{Smooth Ricci potential and the monotone condition} \label{MAsection}

Lemma \ref{l:monotone} above, shows that the angle constraints $\beta\in \betaCONE$, expressed by Equation \eqref{anglecondition}, is a necessary condition for the existence of a Calabi-Yau cone metric on \(X\) with cone angles \(2\pi\beta_a\) along its toric divisors \(D_a\). Coming up next, we strength this result showing that the angle constraints $\beta\in \betaCONE$ is equivalent to the existence of a smooth potential for the associated Ricci form on $X$. Precisely we have the following. 

\begin{proposition} \label{prop:RICPOT}
	Let \(G \in \mS_{\xi}(C, \ell, \beta)\) define a toric K\"ahler cone metric on \((X, \omega)\) with cone angles \(2\pi\beta_a\) along its toric divisors \(D_a\) and let \(\rho\) be its associated Ricci form. There is a function \(h\), smooth outside the apex, which satisfies \(\rho = i \dd h\) on \(X_0 := X \setminus \cup_a D_a\) if and only if $\beta\in \betaCONE$.
\end{proposition}

We note here that \(h\) assumed to be smooth with respect to the differential structure of the fixed smooth symplectic manifold \((X, \omega)\) and \(\rho\) is the Ricci form of the K\"ahler structure defined by \(G\) via . The argument of the next proof goes along the lines of \cite[Proposition 6.8]{FOW}.

\begin{proof}
    Since \(\rho\) is invariant under \(\bT\) and the \(\R^+\) action by dilations, that is \(\mathcal{L}_{\p_{\theta_i}}\rho = 0 \) for all \(i\) and \(\mathcal{L}_{r\p_r} \rho =0 \), we can assume that the same holds for \(h\).
	On the pre-image of the interior of the moment cone, \(X_0 = \mu^{-1}(C_0) \cong (\C^*)^{n+1} \), we have a K\"ahler
	potential \(F\) given by the Legendre transform of the symplectic potential \(G\). The Ricci form is
	\begin{align*}
		\rho &= - i \dd \log \det F_{ij} \\
		&= i \dd h .
	\end{align*}

	We use logarithmic complex coordinates \(z_j = \log w_j = y_j + i \theta_j\) with \((w_0, \ldots, w_n) \in (\C^*)^{n+1} \cong X \setminus \cup_a D_a \). Any \(\bT\)-invariant pluri-harmonic function on \((\C^*)^{n+1}\) is an affine function of \((y_0, \ldots, y_n)\). Therefore, up to subtracting a constant from \(h\), we have
	\begin{equation*}
	\log \det F_{ij} = - 2\gamma_i y_i - h
	\end{equation*}
	with \(\gamma = (\gamma_0, \ldots, \gamma_n) \in \R^{n+1}\). Equivalently,
	\begin{equation} \label{MAeq}
	\det G_{ij} = \exp \left(2\gamma_i \frac{\p G}{\p x_i} + h \right) .
	\end{equation}
	Taking the derivative with respect to \(x_j \p_{x_j}\), we get
	\begin{equation*}
	\langle \gamma, \xi \rangle = -n-1 .
	\end{equation*}

	Write
	\[G = \frac{1}{2}\sum_{a=1}^d \beta_a^{-1} \ell_a \log \ell_a - \frac{1}{2}\ell_\infty \log \ell_\infty + \frac{1}{2}\xi\log \xi  + f \]
	as in Corollary \ref{p:MSYprop}.

	\begin{enumerate}
		
		\item[(i)] Incorporating the boundary behaviour \eqref{eq:GUILLp1}, we see that \(\exp \left(2\gamma_i \frac{\p G}{\p x_i} + h \right)\) equals (up to a constant factor) 
		
		\[\prod_a (\ell_a/\ell_{\infty})^{\beta_a^{-1}\langle v_a, \gamma \rangle} \ell_{\xi}^{-n-1} \exp \left(2\gamma_i \frac{\p f}{\p x_i} + h \right) ; \]
		
		\item[(ii)] and \(\det G_{ij} = f_0 \prod_a \ell_a^{-1} \) with \(f_0\) smooth on \(C\) minus the apex, see\cite{abreuOrbifold}. \footnote{Indeed \(f \in \mathcal{H}(d-n-1)\) meaning that is homogeneous of degree \(d-n-1\).}
		
	\end{enumerate}
	
	It follows from Equation \eqref{MAeq} together with (i) and (ii), that
	\begin{equation}
	\langle v_a, \gamma \rangle = - \beta_a \hspace{2mm} \mbox{for all } a .
	\end{equation}
	We see that \(p := -\gamma\) belongs to the interior of \(C\) and Equation \eqref{anglecondition} holds.
	
	Conversely, assume that \(\beta \in \betaCONE\) and \(G \in \mS_{\xi}(C, \ell, \beta)\). Same as before,
	\(\det G_{ij} = f_0 \prod_a \ell_a^{-1}\) with \(f_0\) smooth and positive on \(C\). As usual,
	\begin{align*}
		\rho &= i \dd \log \det (G_{ij}) \\
		&= i \dd \left(S_1 - \sum_a \log \ell_a \right)
	\end{align*}
	 with \(S_1= \log f_0\) smooth on \(C\). On the other hand,
	 we have complex coordinates \((y, \theta)\) for the K\"ahler structure defined by \(G\), where \(y_i= \p G/ \p x_i\). Given \(\gamma = (\gamma_0, \ldots, \gamma_n) \in \R^{n+1}\), we have a pluriharmonic function on \(X_0\) 
	\begin{align*}
		2\langle y, \gamma \rangle &= 2 \frac{\p G}{\p x_i} \gamma_i \\
		&= \sum_a \beta_a^{-1} \langle v_a, \gamma \rangle \log \ell_a + S_2 ,
	\end{align*}
	where \(S_2\) is smooth on \(C\). If we take \(\gamma=p\), so \(\langle v_a, \gamma \rangle = \beta_a\), we see that \(S_2+\sum_a \log \ell_a\) is pluri-harmonic on \(X_0\). We conclude that \(\rho= i \dd h\) on \(X_0\) with \(h=S_1+S_2\) smooth on \(C\).
	
\end{proof}

In the setting of Porposition \ref{prop:RICPOT}, the volume form \(\omega^{n+1}/(n+1)!\) defines an Hermitian metric on \(K_X\) (outside the apex, of course) regarded as an holomorphic line bundle with the complex structure determined by \(G\). This Hermitian metric is smooth on \(X \setminus \cup_a D_a\) and \(|dz_0 \wedge \ldots \wedge dz_n|^2 = (\det F_{ij})^{-1}\). The volume form \(e^{-h} \omega^{n+1} / (n+1)! \) defines a flat smooth Hermitian metric on \(K_X|_{X\setminus \cup_a D_a}\), which extends only continuously (w.r.t. to the complex atlas detrmined by \(G\)) over the invariant divisors outside the apex.
On \(X \setminus \cup_a D_a\), we can write a locally defined unitary section of \(K_X\) by
\begin{equation*}
	\Omega = e^{i\alpha} e^{-h/2} (\det F_{ij})^{1/2} dz_1 \wedge \ldots \wedge dz_n ,
\end{equation*}
where \(\alpha\) is a real valued function of the arguments \((\theta_0, \ldots, \theta_n)\).
If we set \(\alpha = - \sum_i \gamma_i \theta_i\), then
\begin{align*}
\Omega &= e^{-\sum_i \gamma_i z_i} dz_1 \wedge \ldots \wedge dz_n \\
&= w_1^{-1-\gamma_1} \ldots w^{-1-\gamma_n} dw_1 \wedge \ldots \wedge dw_n 
\end{align*}
is a unitary holomorphic (or equivalently, parallel) section of \(K_X\). Up to a constant factor, the volume form of any Ricci flat K\"ahler cone metric on \(X\) with cone angles \(2\pi\beta_a\) and Reeb vector \(\xi\) as above, is given by \(\Omega \wedge \bar{\Omega}\).

Let \(\rho^T\) denote the transverse Ricci form, it is related to the Ricci curvature of the cone via
\begin{equation}\label{eq:TRANSVRICCI}
	\rho = \rho^T - (n+1) d\eta .
\end{equation}
The existence of a Ricci potential as above, where \(h\) is a smooth function on \(C\), homogeneous of degree \(0\) and basic, that is
\[r\p_r h = 0, \hspace{2mm} \xi(h) =0 ; \]
is equivalent -as follows from Equation \eqref{eq:TRANSVRICCI}- to \(h\) being a transverse Ricci potential \(\rho^T = (n+1) d\eta + i\p_B\bar{\p}_B h\) on the complement of the toric submanifolds \(\Sigma_a \subset S\) of the Sasakian link. Here \(\p_B\) denotes the basic \(\p\)-operator, see \cite{FOW}. Globally, we have the current equation
\[\rho^T - \sum_a (1-\beta_a)[\Sigma_a] = (n+1) d\eta + i \p_B\bar{\p}_B h  \]
on \(S\).
The existence of such a Ricci potential \(h\) is guaranteed if we assume that \(c_1^B - \sum_a (1-\beta_a) [\Sigma_a]_B \in \R \cdot [d\eta]_B \subset H^2_B(S)\). Indeed after a \(D\)-homothetic transformation one can always assume that the constant multiple of \([d\eta]_B\) is equal to \(n+1\). Finally, we note that \(c_1^B - \sum_a (1-\beta_a) [\Sigma_a]_B \in \R \cdot [d\eta]_B \subset H^2_B(S)\) is equivalent to \(c_1(H) = \sum_a (1-\beta_a) [\Sigma_a]\) as cohomology classes in \(H^2(S, \R)\).

The content of Proposition \ref{prop:RICPOT} can also be interpreted in terms of \emph{monotone labelled cones} as mentioned in Remark~\ref{r:monotone}. More precisely, we let $(C,\tilde{\ell})$ with labelling $\tilde{\ell_a}:=\beta_a^{-1}\ell_a$. We say that $(C,\tilde{\ell})$ is monotone if any of the following (equivalent) conditions holds
\begin{itemize}
	\item[(i)] there exists a ray $E_{\tilde{\ell}} \subset C$ such that $\forall y\in E_{\tilde{\ell}}$, $$ \langle y, \tilde{\ell}_1\rangle = \cdots = \langle y, \tilde{\ell}_d\rangle;$$ 
	\item[(ii)] the inward normals $\tilde{\ell}_1,\dots, \tilde{\ell}_d$ are contained in an affine hyperplane of $\kt$;
	\item[(iii)] $\exists \xi \in C_0^*$ such that $(P_\xi,\tilde{\ell})$ is a monotone labelled polytope;
	\item[(iv)] $\forall \xi \in C_0^*$, $(P_\xi,\tilde{\ell})$ is a monotone labelled polytope.
\end{itemize}
With this notion, the condition \(\beta \in \betaCONE\) is equivalent to $(C,\tilde{\ell})$ being monotone.

\subsection{Algebraic point of view} \label{sect:ALG}
We setup some standard algebraic geometry notation, following Cox-Little-Schenck's book \cite{Cox}.
Our toric K\"ahler cone \(\bar{X}\) is isomorphic to a complex affine variety
\begin{equation*}
\bar{X} = U_{\sigma} = \mbox{Spec}(\C[S_{\sigma}]) .
\end{equation*}
Here, \(\sigma = C^*\) and \(S_{\sigma}\) is the semigroup given by \(\Lambda^* \cap C\) where \(\Lambda^*\) is the dual lattice to the kernel of the exponential map \(\Lambda = \ker (\exp) \subset \mathfrak{t}\).\footnote{The standard algebraic notation (\cite{Cox}) is \(\Lambda = N\), \(\Lambda^* = M\). So \(\kt = N \otimes \R\), \(\kt^*= M \otimes {\R}\) and our complexified torus \(\bT \otimes \C\) agrees with \(T_N = (N \otimes \C) / N\).} 
Our affine toric variety \(\bar{X}\) corresponds to the fan that consists of the single cone \(\sigma\) and all of its faces. 

\begin{remark}
	This is the only section of the paper with purely algebro-geometric content.
	For simplicity of notation, in this section we will write \(X\) for the cone with the apex. Similarly, we also write \(D_a\) for the invariant Weil divisors, so \(o \in D_a\).
\end{remark}

\(X\) is a normal affine variety with an isolated singularity at \(\{o\}\), so \(K_X\) is well-defined as a Weil divisor, see \cite[Definition 8.0.20]{Cox}. We also have the toric Weil divisors \(D_1, \ldots, D_d\). The basic fact, \cite[Theorem 8.2.3]{Cox}, is that 
\[K_X = -\sum_{a=1}^d D_a , \]
meaning that  \(K_X + \sum_{a=1}^{d} D_a\) is the principal divisor of a meromorphic function. In particular, \(K_X + \sum_{a=1}^{d} D_a\) is always Cartier.

\begin{proposition} \label{CartierProp}
	The angle constraint expressed by Equation \eqref{anglecondition} is equivalent to the log canonical divisor
	\[K_X + \sum_{a=1}^{d}(1-\beta_a)D_a \]
	being \(\R\)-Cartier.
\end{proposition}

\begin{proof}
	Given real coefficients \(c_a\), the \(\R\)-divisor \(E = \sum_{a=1}^{d} c_a D_a\) is \(\R\)-Cartier if and only if, there is \(p \in \kt^*\) such that \(c_a = -\langle p, v_a \rangle\), where \(v_a \in \Lambda \subset \kt\) are the vectors corresponding to \(D_a\) in the fan description; see \cite[Theorem 4.2.8]{Cox}. In our notation, \(\ell_a = \langle \cdot, v_a \rangle\) are the linear  functions defining the facets of the moment cone \(C\) of \(X\) and \(\Lambda^* \otimes \R = \mathfrak{t}^*\). We are interested in the case where
	\begin{align*}
	E &= K_X + \sum_{a=1}^{d} (1-\beta_a) D_a \\
	&= - \sum_{a=1}^{d} \beta_a D_a .
	\end{align*}
	We conclude that the divisor \(E\) is \(\R\)-Cartier if and only if there is \(p \in \mathfrak{t}^*\) such that
	\(\beta_a = \ell_a(p)\) for every \(a=1, \ldots, d\). Since \(\beta_a>0\) we must necessarily have that \(p\) belongs to the interior of \(C \subset \mathfrak{t}^*\).
\end{proof}

Assume that \(K_X  + \sum_a (1-\beta)_a D_a \) is \(\R\)-Cartier and write \(\Delta = \sum_a (1-\beta)_a D_a\), then  \((X, \Delta)\) is said to be a log pair. Let \(\pi: Y \to X\) be a log smooth resolution, that is a proper birational morphism with \(Y\) smooth and \(\mbox{Exc}(\pi) \cup_a \pi^{-1}(D_a)\) a simple normal crossing divisor. Write \(\Delta'\) for the proper transform of \(\Delta\), that is \(\Delta' = \sum_a (1-\beta)_a D_a'\) with \(D_a'\) equal to the closure of \(\pi^{-1}(D_a \cap X^{reg})\), with \(X^{reg} = X \setminus \{o\}\). Write \(E_i\) for the irreducible divisors lying on the exceptional locus of \(\pi\), so \(\mbox{Exc}(\pi) = \cup_i E_i\). Since \(K_X + \Delta\) is \(\R\)-Cartier, we can pull-it back to the resolution and write
\begin{equation}
K_Y + \Delta' = \pi^*(K_X + \Delta) + \sum_i a_i E_i ,
\end{equation}
for some \(a_i \in \R\). The numbers \(a_i\) are the (log) discrepancies. The pair \((X, \Delta)\) is klt if the discrepancies satisfy \(a_i>-1\) for all \(i\). \footnote{We have assumed that \(\beta_a>0\) for all \(a\), so the coefficients of \(\Delta\) are always strictly less than one, and this avoids the log canonical case.}

\begin{proposition} \label{kltProp}
	If \(K_X  + \sum_a (1-\beta)_a D_a \) is \(\R\)-Cartier, then the pair \((X, \sum_a (1-\beta_a)D_a)\) is klt.
\end{proposition}

\begin{proof}
	Recall that, the fan of \(X\) consists of the single cone \(\sigma\) and all its faces, with \(\sigma = \mbox{Cone}(v_1, \ldots, v_d) \). The generating rays of \(\sigma\) are \(v_1, \ldots, v_d\) and we write \(\sigma(1) = \{v_1, \ldots, v_d\} \).
	We take a toric log smooth resolution \(\pi: Y \to X \). The variety \(Y\) is given by a fan \(\Sigma\) which is obtained by adding some vectors \(v'\) in the interior of the cone \(\sigma\). That is, the generating rays of \(\Sigma\) are given by \(v_1, \ldots, v_d\) together with some vectors \(v_i'\) lying on the interior of \(\sigma\). We write \(\Sigma(1) = \sigma(1) \cup \{v_i'\} \). The map \(\pi\) corresponds to the inclusion map of fans of \(\Sigma\) into \(\sigma\) and the \(v_i'\) correspond to the irreducible components \(E_i \subset Y\) of the exceptional divisor. 
	
	Let \(p \in C_0 \subset \mathfrak{t}^*\) be as in the proof of Proposition \ref{CartierProp}, so \(\langle p, v_a \rangle = \beta_a\) for \(a=1, \ldots, d\). 
	We have
	\begin{align*}
	\pi^*(K_X + \Delta) &= -\sum_{u \in \Sigma(1)} \langle p, u \rangle D_u \\ &= K_Y +  \sum_{u \in \Sigma(1)} (1-\langle p, u \rangle) D_u \\
	&= K_Y + \tilde{\Delta} + E -  \sum_{v'} \langle p, v' \rangle D_{v'} ;
	\end{align*}
	where \(\Delta' = \sum_{u \in \sigma(1)} (1-\langle p, u \rangle)) D_{u}\) is the strict transform of \(\Delta = \sum_{a=1}^d (1-\beta_a) D_a\) and \(E = \sum_{v'} D_{v'}\) is the exceptional divisor. Therefore,
	\begin{equation*}
	K_Y + \Delta' + E = \pi^*(K_X + \Delta) + \sum_{v'} \langle p, v' \rangle D_{v'} .
	\end{equation*}
	For each \(v' \in \Sigma(1) \setminus \sigma(1)\), we write \(v' = \sum_a \lambda_a v_a\) with \(\lambda_a \geq 0\). So, \(\langle p, v'\rangle = \sum_a \lambda_a \beta_a\). Since \(\beta_a>0\) for all \(a\), \(\lambda_a \geq 0\) for all \(a\) and not all \(\lambda_a\) vanish, we see that \(\langle p, v'\rangle >0\) for all \(v'\). We conclude that
	\begin{equation*}
	K_Y + \Delta' = \pi^*(K_X + \Delta) + \sum_{v'} a_{v'} D_{v'} ,
	\end{equation*}
	with \(a_{v'} = \langle p, v' \rangle -1 > -1\). Hence, the pair \((X, \Delta)\) is klt.
\end{proof}

Proposition \ref{kltProp} is a pair version of the well known fact that toric \(\Q\)-Gorenstein singularities are automatically klt.

\begin{remark}
	It follows from Proposition \ref{CartierProp} and Proposition \ref{kltProp} that, if \(\beta = (\beta_1, \ldots, \beta_d) \) belongs to the \(n+1\)-dimensional polytope \(\betaCONE \cap (0,1)^d\), where \(\betaCONE \subset \R^d\) is the angles' cone, then the pair \((X, \Delta)\)  is klt and \(\Delta = \sum_a (1-\beta_a)D_a\) is effective. 
	
	The above is an affine analogue of the fact that projective toric varieties are of `log Fano type', see \cite[Example 11.4.26]{Cox}. Indeed, let \(Z\) be a projective toric variety, for simplicity, assume it is smooth. Let \(L\) be an ample line bundle on \(Z\), so it gives a Delzant lattice polytope. Multyplying by a sufficiently large constant and translating we can assume that the origin is an interior point of the polytope, and this implies that \(L=\sum_{\rho}a_{\rho}E_{\rho}\) where \(E_{\rho}\) are the toric divisors, and \(a_{\rho}>0\). If \(\epsilon>0\) is small so that \(\epsilon a_{\rho} < 1\) for all \(\rho\), then \(\sum_{\rho}(1-\epsilon a_{\rho})E_{\rho} \) is effective and the pair \((Z, \sum_{\rho}(1-\epsilon a_{\rho})E_{\rho})\) is log Fano. The last assertion follows from the following
	\begin{align*}
	c_1(Z) - \sum_{\rho}(1-\epsilon a_{\rho})c_1(E_{\rho}) &= \epsilon \sum_{\rho} a_{\rho}c_1(E_{\rho}) \\
	&= \epsilon c_1(L) \\
	&> 0 .
	\end{align*}
	
\end{remark}

\subsection{Sasakian point of view}

 We denote by \([\Sigma_a] \in H^2(S, \R) \)  the Poincar\'e duals of the toric sumbanifolds \(\Sigma_a \subset S\) and by \(H = \xi^{\perp} = \ker(\eta)  \subset TS\)  the contact distribution, with first Chern class \(c_1(H) \in H^2(S, \R) \). 

\begin{proposition} \label{logc1PROP}
	 The angle constraint, given by Equation \eqref{anglecondition} are equivalent to
	\begin{equation}\label{logc1}
	c_1(H) = \sum_a (1-\beta_a)[\Sigma_a]
	\end{equation}
	as de Rham co-homology classes in \(H^2(S, \R)\).
	
\end{proposition}

\begin{proof}
	The main ingredient is the following exact sequence, see \cite[Equation 7.2.1]{BGbook},
	\begin{equation}\label{exactsequence}
	H^0_B(S) \overset{\alpha}{\to} H^2_B(S) \overset{\imath}{\to} H^2(S, \R) \to H^1_B(S)
	\end{equation}
	where \(\alpha(a) = a [d\eta]_B \) and \(\imath [\cdot]_B = [\cdot] \). On the other hand \(H^1_B(S) \cong H^1(S, \R) \), see \cite[Proposition 7.2.3, item (v)]{BGbook}. A well known result of Lerman asserts that the fundamental group of \(S\) is finite, \(\pi_1(S) = \mbox{span}_{\Z}\{v_1, \ldots v_d\} / \Z^{n+1}\), so \(H^1(S, \R)\) vanishes and the last term in the sequence \eqref{exactsequence} is \(0\). We can then split the sequence and write
	\[ H^2_B(S) = H^2(S, \R) \oplus \R \cdot [d\eta]_B . \]
	 We recall  that \(\imath (c_1^B) = c_1 (H)\), see \cite[proof of Proposition 4.3]{FOW}. It follows that Equation \eqref{logc1} is equivalent to
	\begin{equation} \label{logc1-2}
	c_1^B - \sum_a (1-\beta_a)[\Sigma_a]_B \in \R \cdot [d\eta]_B .
	\end{equation}
	On the other hand, it follows from the formula for the transverse Ricci curvature, that
	\[c_1^B = \sum_a[\Sigma_a]_B ,\]
	which is a Sasakian analogue to a well-known toric formula. It follows that Equation \eqref{logc1-2} is equivalent to
	\begin{equation}
	\sum_a \beta_a [\Sigma_a]_B \in \R \cdot [d\eta]_B .
	\end{equation}
	Consider the linear map \(L\) from the vector space \(\oplus_a \R \cdot [\Sigma_a] \cong \R^d\) to \(H^2_B / \R \cdot [d\eta]_B \cong \R^{d-n-1}\) that sends \((\alpha_1, \ldots, \alpha_d)\) to \(\Pi \circ (\sum_a \alpha_a [\Sigma_a]_B) \) where \(\Pi : H^2_B \to H^2_B / \R \cdot[d\eta]_B \) is the quotient projection. The conclusion is that Equation \eqref{logc1} is equivalent to \((\beta_1, \ldots, \beta_d)\) belong to the kernel of the above linear map. Because the linear map is surjective, the dimension of its kernel is equal to \(n+1\). 
	
	Recall that \(\bar{\p} \p \log \ell_a \) can be interpreted as the curvature of a smooth Hermitian metric on the line bundle associated to \(D_a \subset X\), see \cite{guillMET}. (Here we are working on the outside the apex, so \(X\) and \(D_a\) are smooth.) We can restrict this line to a complex line bundle \(L_a\) on \(S\). The pull-back of \(\frac{i}{2\pi} \dd \log \ell_a \) by the inclusion \(S \subset X\) gives a representative of \( c_1(L_a) \cong [\Sigma_a] \in H^2(S, \R) \), see \cite{FOW} for a reference on Chern classes of basic complex vector bundles. On the other hand, if \(p \in \kt^*\) then 
	\begin{equation} \label{eq:DDBAR}
	\sum_a \langle p, v_a \rangle \dd \log \ell_a = 0 
	\end{equation}
	on \(X\). Indeed, we let
	\(G = \frac{1}{2} \sum_a \ell_a \log \ell_a\), with \(\ell_a = \langle v_a, \cdot \rangle \); so
	\begin{align*}
	y &= DG(x) \\ 
	&=\frac{1}{2} \sum_a v_a \left( 1 + \log \ell_a(x) \right) .
	\end{align*}
	Given some fixed \(p \in \R^{n+1}\), the affine linear function of the \(y\) coordinates \(\langle p, y \rangle \) is given by taking inner product with \(p\) in the above equation:
	\begin{equation*}
	2 \langle p, y \rangle = \sum_a \langle p, v_a \rangle \log \ell_a(x) + C ,
	\end{equation*}
	with constant \(C = \sum_a \langle p, v_a \rangle \). In particular, since affine linear functions of \(y\) are pluri-harmonic, \eqref{eq:DDBAR} holds.
	Restricting Equation \eqref{eq:DDBAR} to \(S\), we see that \(\sum_a \beta_a [\Sigma_a] = 0\) in \(H^2(S, \R)\) (equivalently \(\beta \in \ker(L)\)) if there is some \(p \in \kt^*\) such that \(\beta_a = \ell_a(p)\) for \(a=1, \ldots, d\).
	By dimension counting, we conclude that if \(\beta \in \ker(L)\) then there must be some \(p \in \mathfrak{t}^*\) such that \(\beta_a=\ell_a(p)\) for  all \(a=1, \ldots, d\). Hence, the lemma follows.
	

\end{proof}

\begin{remark}
	The above proof shows that Equation \eqref{logc1} is equivalent to \(\sum_a \beta_a [\Sigma_a]=0\) in \(H^2(S, \R)\).
	It is well known that if \(M\) is a compact toric manifold, then \(H^2(M, \R)\) is generated by the Poincar\'e duals of its toric divisors subject to the relations \(\sum_a \langle v_a, p \rangle [D_a]=0\) for every \(p \in \mathfrak{t}^*\). The above argument gives a  Sasakian analogue, showing that \(H^2(S, \R)\) is generated by the Poinacr\'e duals of \(\Sigma_a\) subject to \(\sum_a \langle v_a, p \rangle [\Sigma_a]=0\) for every \(p \in \mathfrak{t}^*\). 
\end{remark}

We now prove that if the logarithmic first Chern class of the contact distribution vanishes, then the logarithmic basic first Chern class is automatically positive. This is a special feature of the toric set up, and it can be considered as a Sasakian analogue of Proposition \ref{kltProp}.

\begin{proposition} \label{c1basicPROP}
	If Equation \eqref{logc1} holds, then
	\begin{equation}
			c_1^B - \sum_a (1-\beta_a)[\Sigma_a]_B >0 .
	\end{equation}
\end{proposition}

\begin{proof}
	Combining the exact sequence \eqref{exactsequence}, together with Equation \eqref{logc1}, we get that
	\[	c_1^B - \sum_a (1-\beta_a)[\Sigma_a]_B = \tau [d \eta]_B ,\]
	for some \(\tau \in \R\). The fact that \(\tau>0\) follows by taking the wedge product with \((d\eta)^{n-1} \wedge \eta\) integrating over \(S\). Up to positive dimensional factors, we have: 
	\begin{itemize}
		\item The l.h.s. is the transverse scalar curvature, which is positive and given by the volume of the boundary of the cross section polytope determined by the Reeb vector.
		\item The r.h.s. is \(\tau \mbox{vol}(S)\) .
	\end{itemize}
	We deduce from the above two bullets that \(\tau>0\).
	
	Alternatively, we can also argue as follows. If Equation \eqref{anglecondition} holds, then the compact polytope \(P_{\xi}\) labelled by \(\beta_a^{-1}\ell_a\) is monotone. The constant \(\tau\) is given by evaluating \(\beta_a^{-1} \ell_a\) at some interior point of \(P_{\xi}\), so this number can only be positive.
	
\end{proof}

\section{Examples}

In this section we provide examples and compare our work with previously known results about  K\"ahler-Einstein metrics on toric manifolds with conical singularities, \cite{RKE_legendre}, \cite{Datar}.

We make explicit the cone angle constraints expressed by Equation \eqref{anglecondition}.
In order to determine the beta cone \(\betaCONE \subset \R^d\), in the general context of Theorem \ref{MAINTHM}, we consider the linear map \(\R^{n+1} \to \R^d\) given by the \(d\times(n+1)\) matrix \(A\) whose \(d\) rows are the vectors \(v_1^T, \ldots, v_d^T\). Then \(\betaCONE = \mbox{Image}(A) \cap \R^d_{>0}\). In practice, we find a basis \(\{\eta_1, \ldots, \eta_{d-n-1}\}\) for the kernel of \(A^T = (v_1, \ldots, v_d)\) and then \(\betaCONE\) is the intersection of the positive octant \(\R^{d}_{>0}\) with the linear space given by \(\langle \beta, \eta_i \rangle =0\) for \(i=1, \ldots,d-n-1\). The condition that \(\bar{X}\) admits a smooth Calabi-Yau cone metrc, that is the affine toric singularity \(\bar{X}\) is \(Q\)-Gorenstein, is equivalent to \(\R_{>0} \times (1, \ldots, 1) \in \betaCONE\), in other words the image of \(A\) contains the ray where all the entries are equal.

\subsection{Cones over projective toric varieties}
Let \((M, L)\) be a compact polarized toric manifold, so it corresponds to a lattice Delzant polytope 
\[ P = \cap_{a=1}^d \{ \bl_a \geq 0 \} \subset \R^n , \]
where \(\bl_a = \langle \bv_a, \cdot \rangle + \lambda_a\) with \(\bv_a \in \Z^n\) the primitive inward normals. Under this correspondence,
\[ L = \sum_a \lambda_a \bD_a , \] 
where \(\bD_a\) are the toric divisors corresponding to \(\bl_a\). We consider the cone over \(P\), that is the affine toric manifold \(\bar{X}\) whose moment cone  is
\begin{equation*}
C = \{(p, s) \in \R^{n+1} |\, s \geq 0, \,\,  p \in sP \} .
\end{equation*}
As an algebraic variety, \(\bar{X}\) is the Spec of the \(\C\)-algebra of the semi-group \(C \cap \Z^{n+1}\), and it is isomorphic to the total space of the dual of \(L\) with its zero section contracted to a point, \((L^*)^{\times}\). The facet normals to \(C\) are
\[v_a= (\bv_a, \lambda_a) . \]
A well-known family of examples is when \(P\) is the anit-canonical polytope of a Fano variety, in which case \(P\) is reflexive and we can take \(\lambda_a=1\) for all \(a\). Then, the singularity \(\bar{X}\) is Gorenstein and it admits a smooth Calabi-Yau cone metric by \cite{FOW}. But in general, there is no reason for the vectors \(v_a\) to lie on a hyperplane, which correspond to the \(\Q\)-Gorenstein condition, necessary to admit  smooth Calabi-Yau cone metrics.

For any polytope \(P\) as above, not necessarily reflexive, there is one parameter family of toric K\"ahler-Einstein metrics \(\omega_t\), for \(t>0\), with cone angles \(2\pi t \beta_a\) along \(\bD_a\) and
\begin{equation*}
\beta_a = \bl_a (\hat{p}) , \hspace{2mm} \mbox{with } \hat{p} = \mbox{Barycenter}(P) .
\end{equation*}
The K\"ahler-Einstein equation is
\begin{equation*}
\mbox{Ric}(\omega_t) = \alpha_t \omega_t + \sum_a (1- t\beta_a)[\bD]_a .
\end{equation*}
We recall that, up to \(2\pi\) factors, the class of \(\mbox{Ric}(\omega_t)\) is \(c_1(M)\) and it is equal to \(\sum_a [\bD]_a\). On the other hand, for any \(p \in \R^2\) we have \(\sum_a \bl_a(p) [\bD]_a = c_1(L) \). We conclude that \(\alpha_t = t\). The supremum of \(t>0\) such that \(t\beta_a \leq 1\) for all \(a\) is the invariant \(R(M, L)\) studied by Datar-Guo-Song-Wang \cite{Datar}. We can lift these conical K\"ahler-Einstein metrics on the polytope \(P\) to regular Calabi-Yau cone metrics on \(\bar{X}\) with cone angles \(2\pi\beta_a\) along \(D_a\). Recall that \(\xi \in \mathfrak{t}_+\) defines a cross section \(\{\langle y, \xi \rangle = n+1\}\), so this cross section is equal to \(P \times \{1\}\) for \(\xi = (0, n+1)\). Note that \(\R_{>0} \cdot (0,1) \subset \mathfrak{t}_+\). The Reeb/cone angle correspondence restricted to this ray of regular structures is then
\[ t(\beta_1, \ldots, \beta_d) = t (\bl_1(\hat{p}), \ldots, \bl_d(\hat{p})) \leftrightarrow t^{-1}(0,n+1) .\]

\subsection{Three dimensional cones}
Consider first the case when \(\bar{X} = (K_M)^{\times}\) is the canonical bundle of a smooth del Pezzo surface \(M\) with its zero section contracted at the apex. 
The corresponding \(M\) is either \(\mathbb{CP}^2\), \(\mathbb{CP}^1 \times \mathbb{CP}^1\) or the blow ups of the projective plane at either one, two or three points: \(dP_1, dP_2, dP_3\). All the corresponding \(\bar{X}\) admit smooth Calabi-Yau cone metrics, the structure is irregular in the \(dP_1\) and \(dP_2\) cases and a regular lift of a smooth K\"ahler-Einstein metric on the base in the remaining cases. Since \(K_{\bar{X}}\) is Cartier in all these cases, the beta cone is detrmined by the requirement that \(\sum_a \beta_a \bD_a\) is \(\R\)-Cartier. 
The list of reflexive polygons, taken from \cite[Section 8.3]{Cox}, and associated cones goes as follows. 
\begin{enumerate}
	\item \(M = \mathbb{CP}^1 \times \mathbb{CP}^1\), then
	\[\mbox{vertices of } P = \{ (-1,-1), (-1,1), (1,1), (1,-1) \} .\]
	In the following we use \(\sigma = C^*\) and \(\sigma(1)\) is the set of facet normals to \(C\), equivalently, the ray generators of \(\sigma\).
	\[ \sigma(1) = \{ v_1 = ( 1,  0, 1), v_2 =( 0,  1, 1), v_3 =(-1,  0, 1), v_4=( 0, -1, 1) \} \]
	\[\mbox{beta cone } B= \{\beta_1 + \beta_3 = \beta_2 + \beta_4 \}  .\]
	The fan of \(\bar{X}\) is consists of the single cone \(\sigma\) with generating rays \(v_1, \ldots, v_4\). There is a sub-cone \(\sigma' \subset \sigma\) with generating rays 
	\[\{ v_1' = ( 0,  0, 1), v'_2 =( 0,  1, 1), v'_3 =(1,  1, 1), v'_4=( 1, 0, 1) \} . \]
	The toric variety associated to \(\sigma'\) is known as the conifold, or \(T^{1,1}\)-singularity,
	\[ \mathcal{C} = U_{\sigma'} = \{UV=ZW\} \subset \C^4 . \]
	The variety \(\mathcal{C}\) is isomorphic to the total space of \(\mathcal{O}(-1, -1)\) over \(\mathbb{CP}^1 \times \mathbb{CP}^1\) with its zero section contracted to a point. The inclusion \(\sigma' \subset \sigma\) realizes a two-fold covering map and \(\bar{X} = \mathcal{C} / \Z_2\). The base of the conifold, \(S = \mathcal{C} \cap S^7\), is identified with \(S^2 \times S^3\). The K\"ahler-Einstein product metric on \(\mathbb{CP}^1 \times \mathbb{CP}^1\) lifts to a smooth homogeneous Sasaki-Einstein metric, which leads to a regular Calabi-Yau cone on \(\mathcal{C}\), known as the Stenzel metric. More generally, the ray of regular Calabi-Yau cone metrics on \(\mathcal{C}\) with equal cone angle \(2\pi\beta\) along the four toric divisors, are lifts of \(\mathbb{CP}_{\beta} \times \mathbb{CP}_{\beta} \). Here we have used \(\mathbb{CP}_{\beta}\) to denote the Riemann sphere endowed with the rugby ball metric with cone angles \(2\pi\beta\) and polarized by \(\mathcal{O}(1)\) (area = \(2\pi\)), its symplectic potential is
	\begin{equation*}
	\beta^{-1} y \log y + \beta^{-1} (1-y) \log(1-y) .
	\end{equation*}
	
	The link \(S\) of \(\mathcal{C}\) can be identified with the set of unit vectors on \(\mathcal{O}(-1,-1)\) on each factor. Write \(\Sigma_i = D_i \cap S\) for the real codimension two submanifolds of \(S\) given by the intersection of the four toric divisors with the link. Each \(\Sigma_i\) is diffeormorphic to \(S^3\). We have \(H^2(S) = \R\), and each of the Poincar\'e duals of \(\Sigma_i\) is a generator. Moreover \([\Sigma_1] = [\Sigma_3] \), \([\Sigma_2] = [\Sigma_4]\) and \([\Sigma_1] = -[\Sigma_2]\). On the other hand, \(c_1(H) =0\), where \(H \subset TS\) is the contact distribution. We see that the requirement associated to the beta cone condition: \(\beta_1 + \beta_3 = \beta_2 + \beta_4\); is equivalent to \(c_1(H) = \sum_i (1-\beta_i)[\Sigma_i] \).
	
	Up to a constant factor, the volume functional \(\mbox{vol}(\Delta_{\xi})\) in \(\sigma'\), with coordinates \((\xi_1,\xi_2,\xi_3)\), is equal to
	\[\frac{\xi_3}{\xi_1 \xi_2(\xi_3-\xi_2)(\xi_3-\xi_1)} . \]
	Given \(\beta = \beta(p) \in \betaCONE\), the Reeb vector is given by minimizing \(\mbox{vol}(\Delta_{\xi})\) on \(C^* \cap \{\langle \xi, p  \rangle =3\} \).

	\item \(M=dP_1\), the blow up of \(\mathbb{CP}^2\) at one point, then
	\[\mbox{vertices of } P = \{ (-1,-1), (-1,1), (0,1), (2,-1) \} ,\]
	\[ \sigma(1) = \{ v_1 =(1,  0,  1), v_2 =(0, -1,  1), v_3 =(-1, -1,  1), v_4 =(0,  1,  1)
	\} \]
	\[\mbox{beta cone } \betaCONE= \{2\beta_1 + 2\beta_3 = 3\beta_2 + \beta_4 \}  .\]

	The barycenter of \(P\) is located at 
	\[P_c = \frac{1}{4} \left(\frac{1}{3}, -\frac{2}{3} \right) .\]
	The affine linear functions defining \(P\) are
	\[\bl_1 = x + 1, \hspace{1mm} \bl_2 = -y + 1, \hspace{1mm} \bl_3 = -x-y+1, \hspace{1mm} \bl_4 = y + 1 . \]
	Evaluating at the barycenter, \(\beta_a = \bl_a(P_c) \), we have
	\[\beta_1 = \frac{13}{12}, \hspace{1mm} \beta_2 = \frac{7}{6}, \hspace{1mm} \beta_3 = \frac{13}{12}, \hspace{1mm} \beta_4 = \frac{5}{6} . \]
	The vector \(\beta = (13/12, 7/6, 13/12, 5/6) \)  satisfies \(2\beta_2 + 2\beta_3 = 3\beta_2 + \beta_4 \).
	The family of K\"ahler-Einstein metrics \(\omega_t\) on \(M\), which solve
	\[\mbox{Ric}(\omega_t) = t \omega_t + \sum_a (1-t\beta_a)[\bD_a] , \]
	lifts to a family of regular Calabi-Yau cone metrics on \(\bar{X}\) with Reeb vector field and cone angles related by 
	\[ \xi = t^{-1}(0,0,3) \leftrightarrow t \beta . \]
	The `upper Ricci lower bound invariant' \(R=R(M, K_M)\) is the supremum of all \(t>0\) such that the entries of \(t\beta\) are all \(\leq 1\), equivalently \(R = (\max_a \beta_a)^{-1}\). In the case \(M= dP_1\), we easily see that \(R=6/7\). On the other hand, the ray of angles \(\R_{>0} \cdot (1,1,1,1) \) corresponds to a ray of irrational vector fields, see \cite{MSY}, which give rise, when all cone angles are equal to \(2\pi\), to a smooth irregular Sasaki-Einstein metric.

	\item \(M=dP_2\), the blow up of \(\mathbb{CP}^2\) at two points, then
	\[\mbox{vertices of } P = \{ (-1,-1), (-1,1), (0,1), (1,0), (1,-1) \} ,\]
	\[ \sigma(1) = \{ v_1, \ldots, v_5
	\} \]
	with
	\begin{align*}
	v_1 = (1,  0,  1), \hspace{1mm} v_2 = (0, -1,  1), \hspace{1mm} v_3 = (0,  1,  1) \\
	v_4 =  (-1,  0,  1), \hspace{1mm} v_5 = (-1, -1,  1)
	\end{align*}
	\[\betaCONE= \{\beta_1 + 3\beta_4 = 2\beta_3 + 2\beta_5, \hspace{1mm} \beta_2 + 2 \beta_4 = \beta_3 + 2 \beta_5 \}  .\]
	The barycenter of \(P\) is located at 
	\[P_c = \frac{2}{7} \left(-\frac{1}{3}, -\frac{1}{3} \right) .\]
	The affine linear functions defining \(P\) are
	\[\bl_1 = x + 1, \hspace{1mm} \bl_2 = -y + 1,\]
	\[\hspace{1mm} \bl_3 = y+1, \hspace{1mm} \bl_4 = -x + 1, \bl_5 = -x-y+1 . \]
	Evaluating at the barycenter, \(\beta_a = \bl_a(P_c) \), we have
	\[\beta_1 = \frac{19}{21}, \hspace{1mm} \beta_2 = \frac{23}{21}, \hspace{1mm} \beta_3 = \frac{19}{21}, \hspace{1mm} \beta_4 = \frac{23}{21}, \hspace{1mm} \beta_5 = \frac{25}{21} ; \]
	and \(R=21/25\). The same commentaries as in the \(dP_1\) example apply.
	
	\item \(M=dP_3\), the blow up of \(\mathbb{CP}^2\) at three points, then
	\[\mbox{vertices of } P = \{ (-1,-1), (-1,0), (0,1), (1,1), (1,0), (0,-1) \} ,\]
	\[ \sigma(1) = \{ v_1, \ldots, v_6
	\} \]
	with
	\begin{align*}
	v_1 = (1,  0,  1), \hspace{1mm} v_2 = (0, 1,  1), \hspace{1mm} v_3 = (1,  -1,  1) \\
	v_4 =  (-1,  1,  1), \hspace{1mm} v_5 = (-1, 0,  1), \hspace{1mm} v_6=(0,-1,1) .
	\end{align*}
	\[\betaCONE= \{\beta_1 + 3\beta_5 = 2\beta_4 + 2\beta_6, \hspace{1mm} \beta_2 + 2 \beta_5 = 2\beta_4 +  \beta_6, \hspace{1mm} \beta_3 + 2\beta_5 = \beta_2 + 2 \beta_6 \}  .\]
	The barycenter of \(P\) is located at zero and \(M\) admits a family of conical K\"ahler-Einstein metrics \(\omega_t\) with cone angles \(t\cdot(1, \ldots,1)\), smooth at \(t=1\). The \(\omega_t\) lift to regular Calabi-Yau cone metrics on \(X\) with cone angles \(2\pi t\) along the toric divisors and Reeb vector \(t^{-1}(0,0,3)\).
\end{enumerate} 

In the previous examples, we have considered the anti-canonical polarization and \(\bar{X}=(K_M)^{\times}\). However, we can still apply Theorem \ref{MAINTHM} to any polarization \((M,L)\) and \(\bar{X}=(L^{-1})^{\times}\). In general, if \(L\) is not a positive rational multiple of the anti-canonical, then the singularity \(\bar{X}\) is not \(\mathbb{Q}\)-Gorenstein, so its beta cone \(\betaCONE\) does not contain the ray \(\R_{>0} \cdot (1, \ldots, 1)\). Here are a few examples:

\begin{enumerate}
	\item If \(a\) and \(b\) are positive integers, with \(a< b\), we can take \(L=\mathcal{O}(a, b)\) over \(\mathbb{CP}^1 \times \mathbb{CP}^1\). Its moment polytope is the rectangle \(P=[0,a]\times[0,b]\) and \(C = \mbox{Cone}(P \times \{1\})\) is the moment cone of \(\bar{X}\), one easily checks that the facet normals of \(C\) do not lie on a hyperplane. The family of regular structures, with Reeb vector \(t^{-1}\cdot(0,0,3)\) and cone angles \(2\pi t  (a/2, b/2, a/2, b/2) \), are lifts of products of two Rugby ball metrics, \(\mathbb{CP}_{\beta_1} \times \mathbb{CP}_{\beta_2}\) with \(\beta_1 = ta/2\) and \(\beta_2 = tb/2\) and polarizations \(\mathcal{O}(a)\) on the first factor and \(\mathcal{O}(b)\) on the second.

	\item Let \(M\) be the blow up of \(\mathbb{CP}^2\) at one point and take as a moment polytope \(P\) the convex hull of \((1,0), (1,1), (2,2), (2,0)\) with inward normals \(\bv_1 = (1,0), \bv_2 = (1,-1), \bv_3 = (-1,0), \bv_4 = (0,1)\). The Picard group of \(M\) has rank two, it is generated by \(\bD_1, \ldots, \bD_4\) subject to  relations \(\bD_1 + \bD_2 = \bD_3\) and \(\bD_2 = \bD_4\).
	The anti-canonical class is \(-K_M = \sum_a \bD_a = 2 \bD_1 + 3\bD_3\). On the other hand, the polarization on \(M\) associated to \(P\), is \(L= \bD_1 + 2\bD_2\). There is a family of K\"ahler-Einstein metrics \(\omega_t\), in the cohomology class \(2\pi c_1(L)\), with \(\mbox{Ric}(\omega_t) = t \omega_t\) on \(M \setminus \cup_a \bD_a \cong (\C^*)^2\) and cone angles \(t (5/9,7/9,4/9,7/9)\). In particular, at \(t=9/7\) (which is the invariant \(R(M, L)\)), the cone angles are \((5/7,1,4/7,1)\). The metrics \(\omega_t\) can be written explicitly, see \cite[Section 6.4]{RKE_legendre}, and therefore so do their regular lifts to \(X\).
	
\end{enumerate}

There is no need for \(M\) to be Fano, so we can also consider polarized Hirzebruch surfaces \((M, L)\) with \(M = F_a = \mathbb{P}(\mathcal{O} \oplus \mathcal{O}(a))\). In general, \(\bar{X}\) does not have to be \((L^{-1})^{\times}\) with \(L\) an ample line bundle over a smooth compact toric surface. For example, the affine toric varieties \(\bar{X}\) corresponding to the \(Y^{p,q}\) singularities, see \cite{MSY_toric}. These are Gorenstein and admit smooth Calabi-Yau cone metrics, our Theorem \ref{MAINTHM} realizes each of these metrics as a member in a three dimensional family by introducing cone angles.
Following \cite{MSY_toric}, the facet normals to the moment cone of the \(Y^{p,q}\) singularity are
\[ \sigma(1) = \{ v_1 = ( 1,  0, 0), v_2 =( 1,  p-q-1, p-q), v_3 =(1,  p, p), v_4=( 1, 1, 0) \} \]
and  \(B\) given by \((p+q)\beta_1 + (p-q)\beta_3 = p\beta_2 + p \beta_4\).

\subsubsection{Explicit formula for symplectic potentials}
 As a final remark we give a recipe to get explicit toric conically singular Calabi-Yau metric on K\"ahler cone whose moment cone has $4$ facets. Recall that all the K\"ahler--Einstein potentials on convex quadrilateral are explicitely known by Apostolov--Calderbank--Gauduchon \cite{ACG:weakly} and the second author \cite{TGQ}. To write down the explicit formula one needs to put suitable coordinates on the quadrilateral that depends on the type (trapezoids, parallelogram or generic) of the quadrilateral. Thus, given any Calabi-Yau cone metric as in Theorem \ref{MAINTHM} with a four faced good moment cone the associated potential on the tranversal polytope has no choice to fall into the category of metrics studied by \cite{ACG:weakly}.     
On the other hand, we note that any two strictly convex four faced cones in \(\R^3\) are equivalent under \(SL(3, \R)\). Thus, the symplectic potential of a Calabi-Yau cone metric as in Theorem \ref{MAINTHM} with a four faced good moment cone can be realized (on the complement of the invariant divisors) as one of KE potential the family of a fixed cone $C$ (with base a square, say) and various labelling. Let's do an example in detail. To simplify assume the Calabi-Yau metric we started with on $X$ is a smooth one (i.e $\beta_1=\dots=\beta_4=1$) and the moment cone, image of $\mu: X\ra \R^3$, is some four face strictly convex cone $\tilde{C}\subset\R^3$ with inwards primitive normals $w_1, \dots,w_4 \in Z^3$. There is a unique automorphism $\phi\in SL(3, \R)$ sending $\tilde{C}$ to $C := \{ (x_0,x_1,x_2)\in \R^3\,|\, x_0 \pm x_1 \pm x_2 \geq 0\}$ and then $(\phi^{-1})*w_a = r_a v_a$ for some $r_a >0$ where the $v_1,\dots, v_4\in \Z^3$ are the unique inwards primitive normals to $C$. Then the symplectic potential associated to the Calabi-Yau metric on $X$ (with respect to the moment map $\phi\circ \mu$ and associated action-angle coordinates) is the only KE potential in $\mS_\xi(C, \ell,\beta)$  where $\beta=(r_1^{-1},\dots, r_4^{-1}) \in \betaCONE$, $\ell_a=v_a$ and $\xi$ is the Reeb vector field. Moreover, this KE potential can be written down explicitly using Hamiltonian 2-form coordinates~\cite{ACG:weakly,TGQ}.


	


\begin{thebibliography}{DDDD}
	
	
	\bibitem{abreu}
	{\scshape M. Abreu}
	\emph{K\"ahler geometry of toric varieties and extremal metrics}, Internat. J. Math. {\bf 9} (1998), 641--651.
	
 \bibitem{abreuOrbifold}
  {\scshape M. Abreu}
    \emph{K\"ahler metrics on toric orbifolds}, J. Differential Geom. {\bf 58} (2001), 151--187.
    
	\bibitem{Abreu_sas}
	{\scshape M. Abreu}
	\emph{Kahler-Sasaki geometry of toric symplectic cones in action-angle coordinates}
	Portugal. Math. (N.S.), 67 (2010), 121--153. 
	
	
\bibitem{ACG:weakly}
  {\scshape Apostolov, V., Calderbank, D., Gauduchon, P.}
  \emph{The Geometry of Weakly Self-dual K\"ahler Surfaces}. Compositio Mathematica, 135(3), 279-322. (2003)
  
	\bibitem{H2FII}
	{\scshape V. Apostolov, D. M. J. Calderbank, P. Gauduchon, C. T{\o}nnesen-Friedman}
	\emph{Hamiltonian $2$--forms in K\"ahler geometry. II. Global classification}, J. Differential Geom. {\bf 68} (2004), 277--345.
	
	\bibitem{BanyagaMolino}
	{\scshape A. Banyaga,  P.  Molino} 
	\emph{G\'eom\'etrie  des  formes  de  contact  compl\`etement  int\'egrables de type toriques}, S\'eminaire Gaston Darboux de G\'eom\'etrie et Topolo-gie  Diff\'erentielle,  1991--1992  (Montpellier),  Univ.  Montpellier  II,  Montpellier, 1993, pp. 1--25
	
	\bibitem{Berman}
	{\scshape R. Berman,  Bo Berndtsson} 
	\emph{Real Monge-Amp{\`e}re equations and K{\"a}hler-Ricci solitons on toric log Fano varieties}, Annales de la Facult{\'e} des sciences de Toulouse: Math{\'e}matiques,  
	{\bf  4} (2013), 649--711.
	

	\bibitem{BGbook}
	{\scshape C. Boyer and K. Galicki},
	{\it Sasakian geometry},  Oxford Univ. Press (2008).
	
	
	\bibitem{CDG} 
	{\scshape 	D. M. J. Calderbank, L.David, P. Gauduchon},
	{\it The Guillemin formula and K\"ahler metrics on toric symplectic manifolds }    J. Symplectic Geom.
	Volume 1, Number 4 (2002), 767-784.
	
	
	\bibitem{CSz} 
	{\scshape T. Collins and G. Sz\'ekelyhidi},
	{\it $K$-semistability for irregular Sasaki manifolds},  J. Differential Geom. {\bf 109} (2018), 81--109.
	
	\bibitem{Cox}
	{\scshape D. Cox, B. Little and H. Schenck},
	{\it Toric varieties}, Graduate Studies in Mathematics {\bf 124}, AMS, (2011). 
	
	\bibitem{Datar} 
	{\scshape V. Datar, B. Guo, J. Song and X. Wang},
	{\it Connecting toric manifolds by conical K{\"a}hler--Einstein metrics},  Advances in Mathematics {\bf 323} (2018), 38--83.
	
	
	\bibitem{don:scalar}
	{\scshape S. K. Donaldson}
	\emph{Scalar curvature and stability of toric varieties}, J. Differential Geom. {\bf 62} (2002), 289--349.
	
	
	\bibitem{don:Large}
	{\scshape S. K. Donaldson}
	\emph{ K\"ahler geometry on toric manifolds, and some other manifolds with large symmetry}, Handbook of geometric analysis. No. 1, Adv. Lect. Math. (ALM) {\bf 7}, 29--75, Int. Press, Somerville, MA, 2008.
	%
	
	\bibitem{don:extMcond}
	{\scshape S. K. Donaldson}
	\emph{Extremal metrics on toric surfaces: a continuity method}, J. Differential Geom. {\bf 79} (2008), 389--432.
	
	
 \bibitem{futaki}
   {\scshape A. Futaki}
       \emph{An obstruction to the existence of Einstein K\"ahler metrics}, Invent. Math. {\bf 73} (1983), no. 3, 437--443.
       
       
	\bibitem{FOW}
	{\scshape A. Futaki, H. Ono and G. Wang}
	\emph{Transverse K{\"a}hler geometry of Sasaki manifolds and toric Sasaki-Einstein manifolds}, J. Differential Geom. {\bf 83} (2009), 585--636.
	

	
	\bibitem{Gigena}
	{\scshape S. Gigena}
	\emph{Integral invariants of convex cones}, J. Differential Geom. {\bf 13} (1978), 191--222.
	
	\bibitem{Guedj}
	{\scshape V. Guedj and A. Zeriahi}
	\emph{Degenerate Complex Monge--Amp{\`e}re Equations}, EMS Tracts in Mathematics {\bf 26} (2017).
	
	
	\bibitem{guillMET}
	{\scshape V. Guillemin}
	\emph{K\"ahler structures on toric varieties}, J. Diff. Geom. {\bf 40} (1994), 285--309.
	
\bibitem{hashimoto}
   {\scshape Y. Hashimoto}
      \emph{Scalar curvature and Futaki invariant of K\"ahler metrics with cone singularities along a divisor},  arXiv:math.DG/15008.02640v1. 

	
	\bibitem{He-Sun} 
	{\scshape W. He  and S. Sun},
	{\it Frankel conjecture and Sasaki geometry},  Adv. Math. {\bf 291} (2016), 912--960.
	

	\bibitem{Hein-Sun} 
	{\scshape H. Hein  and S. Sun} ,
	{\it Calabi-Yau manifolds with isolated conical singularities}, Publications math{\'e}matiques de l'IH{\'E}S {\bf 126} (2017), 73--130.
	
	
	\bibitem{RKE_legendre}
	{\scshape E. Legendre},
	\emph{Toric K\"ahler-Einstein metrics and convex compact polytopes}, Journal of Geometric Analysis, 26  (1), 2016, 399--427. 
	
	\bibitem{nonUNIQcscS}  
	{\scshape E. Legendre},
	\emph{Existence and non-uniqueness of constant scalar curvature toric Sasaki metrics}, Compositio Math. {\bf 147} (2011), 1613--1634.
	
	
	\bibitem{LTF}  
	{\scshape E. Legendre, C. W. Tønnesen-Friedma}, \emph{Toric Generalized K\"ahler-Ricci Solitons with Hamiltonian 2-form}, Math Z. 274 (2013), pp. 1177--1209.
		
	 \bibitem{TGQ}
  {\scshape E. Legendre}
    \emph{Toric geometry of convex quadrilaterals}, J. Symplectic Geom.   {\bf 9} (2011), pp. 343--385	
	\bibitem{Lerman}  
	{\scshape E. Lerman},
	\emph{Contact toric manifolds}, J. Symplectic Geom. 1(2003), no. 4, 785--828
	
	
	\bibitem{LiLiu} 
	{\scshape C. Li and Y. Liu},
	\emph{K\"ahler-Einstein metrics and volume minimization}, Advances in Math. {\bf 341} (2019), pp. 440--492.
	
	
	
	\bibitem{Li} 
	{\scshape C. Li, Y. Liu and C. Xu},
	{\it A guided tour to normalized volume}    arXiv:1806.07112, (2018).

		
		
	\bibitem{MSY}
	{\scshape D. Martelli, J. Sparks and S.-T. Yau},
	{\it Sasaki--Einstein manifolds and volume minimisation}, Comm. Math. Phys. {\bf 280} (2008), 611--673. 
		
	\bibitem{MSY_toric}
	{\scshape D. Martelli,  J. Sparks, S.-T. Yau}
	\emph{The geometric dual of {\it a}-- maximisation for toric Sasaki-Einstein manifolds}, Comm. Math. Phys.  {\bf 268}  (2006), 39--65.
	
	
	\bibitem{Oda}
	{\scshape T. Oda},
	{\it Convex bodies and algebraic geometry: An introduction to the theory of toric varieties}, Ergebnisse der Mathematik und ihrer Grenzgebiete {\bf 3:15} (1988). 
	
	
	\bibitem{SongWang}
	{\scshape J. Song and  X. Wang}
	\emph{The greatest Ricci lower bound, conical Einstein metrics and Chern number inequality}, Geometry \& Topology.  {\bf 20.1}  (2016), 49--102.
		
	
	\bibitem{schwarz}
	{\scshape G.W. Schwarz}, \emph{Smooth functions invariant under the action of a compact Lie group}, Topology {\bf 14} (1975) 63--68.
	
		
	\bibitem{WangZhu}
    {\scshape X--J. Wang, X.H. Zhu}
     \emph{K\"ahler--Ricci solitons on toric manifolds with positive first Chern class}, Advances in Math. {\bf 188} (2004), 47--103.
		


\end{thebibliography}

\bibliographystyle{abbrv}

\end{document}